\documentclass{amsart}
\usepackage{graphicx}
\usepackage[margin=1in]{geometry} 
\usepackage{amsmath,amsthm,amssymb,amsfonts,mathrsfs}
\usepackage{tikz-cd}
\usepackage{bm}
\usepackage{bbm}
\usepackage{geometry}
\usetikzlibrary{matrix} 
\newtheorem{theorem}{Theorem}[section]
\newtheorem{lemma}[theorem]{Lemma}

\theoremstyle{definition}
\newtheorem{definition}[theorem]{Definition}

\theoremstyle{remark}
\newtheorem{remark}[theorem]{Remark}
\newtheorem{corollary}[theorem]{Corollary}
\newtheorem{proposition}[theorem]{Proposition}
\numberwithin{equation}{section}

\newcommand{\ql}{\mathbb{Q}_{\ell}}
 \newcommand{\res}{\textnormal{res}}
\newcommand{\sht}{\textnormal{Sht}}
\newcommand{\loc}{\textnormal{loc}}
\newcommand{\mb}{\mu_{\bullet}}
\newcommand{\nb}{\nu_{\bullet}}
\newcommand{\vb}{V_{\bullet}}
\newcommand{\wb}{W_{\bullet}}
\newcommand{\hk}{\textnormal{Hk}}
\newcommand{\Hmbnb}{\textnormal{Hk}_{\mu_{\bullet}\mid \nu_{\bullet}}}
\newcommand{\Smbnb}{\sht_{\mu_{\bullet}\mid \nu_{\bullet}}}
\newcommand{\hg}{\textnormal{Hom}_{\hat{G}}}
\newcommand{\cc}{\textnormal{Corr}}
\newcommand{\Gmbnb}{\textnormal{Gr}_{\mu_{\bullet}\mid \nu_{\bullet}}}
\newcommand{\sv}{\mathcal{S}_{V_1,V_2}}
\newcommand{\sch}{S(\lambda)}

\begin{document}

\title{A Geometric Jacquet-Langlands Transfer for automorphic forms of higher weights}

\author{Jize Yu}
\address{Institute for Advanced Study, School of Mathematics,1 Einstein Drive,Princeton, NJ 08540}
\curraddr{Department of Mathematics, Caltech, 1200 East California Boulevard, Pasadena, CA 91125}
\email{jize.yu.math@gmail.com}
\thanks{The author is very grateful to Xinwen Zhu for introducing the beautiful theory of moduli of Shtukas to him, and providing numerous helpful comments and discussions while he was working on this project. }


\subjclass[2020]{Primary 14G35, 11G18; Secondary 14D24, 20C20}

\date{}

\keywords{geometric Satake equivalence, Hodge type Shimura varieties, Jacquet-Langlands Transfer, moduli of local Shtukas}

\begin{abstract}
In this paper, we give a geometric construction of the Jacquet-Langlands transfer for automorphic forms of higher weights by studying the geometry of the mod $p$ fibres of different Hodge type Shimura varieties which satisfy a mild assumption and the cohomological correspondences between them.
\end{abstract}

\maketitle
\tableofcontents

\section{Introduction}
Let $G$ and $G'$ be two algebraic groups over $\mathbb{Q}$. We assume that they are isomorphic at all finite places
of $\mathbb{Q}$, but not necessarily at infinity. Roughly speaking, the Jacquet-Langlands correspondence
predicts, in many cases, the following phenomenon:  there exists a natural map between the set of automorphic
representations of $G$ and that of $G'$ such that if $\pi'$ is the automorphic representation of $G'$
which corresponds to an automorphic representation $\pi$ of $G$, then $\pi_{\nu}$ is isomorphic to $\pi'_{\nu}$ at all finite
places $\nu$. It is considered as one of the first examples of the Langlands philosophy that maps between $L$-groups induce maps between automorphic representations.

The classical way of establishing this correspondence is via a comparison of the trace formulas for
$G$ and $G'$. This approach allows us to conclude a map between suitable
spaces of automorphic forms for $G$ and $G'$ as abstract representations. However, the resulting map is not canonical, and we, therefore, hope to have a more natural way of understanding
this correspondence. Alternative geometric approaches of establishing the Jacquet-Langlands correspondence were first noticed by Ribet \cite{ribet1989bimodules} and Serre \cite{serre1996two}, and later followed by Ghitza \cite{ghitza2003siegel}, \cite{ghitza2005all}, and Helm \cite{helm2010towards} \cite{helm2012geometric}.

In the recent work of Xiao-Zhu \cite{xiao2017cycles}, a more general strategy is developed
to relate the geometry and ($\ell$-adic) cohomology of the mod $p$ fibers of different Shimura varieties,
which unifies and generalizes some aspects of the above mentioned works. In this paper, we adapt the machinery developed in \textit{loc.cit} and apply the integral coefficient geometric Satake equivalence in mixed characteristic obtained by the author \cite{yu2019integral}\footnote{We mention that Peter Scholze announced the same result using the perfectoid approach.} to establish a Jacquet-Langlands transfer for automorphic forms with higher weights.

\subsection{Main theorem}
\subsubsection{Statement of the Main Theorem}Let $(G_1, X_1)$ and $(G_2, X_2)$ be two Hodge type Shimura data equipped with an isomorphism $\theta:G_{1,\mathbb{A}_f}\simeq G_{2,\mathbb{A}_f}$. Assume that there exists an inner twist $\Psi:G_1\rightarrow G_2$ which is compatible with $\theta$. We assume that $K_{i}\subset G(\mathbb{A}_f)$ to be sufficiently small such that $\theta K_1=K_2$. In addition, we assume that $p$ is an unramified prime. Then $K_{1,p}$ (and therefore $K_{2,p}$) is hyperspecial. Let $\underline{G_i}$ be the integral model of $G_{i,\mathbb{Q}_p}$
over $\mathbb{Z}_p$ determined by $K_{i,p}$. Then  $\underline{G_1}\simeq \underline{G_2}$,  and we can thus identify
their Langlands dual groups, which we denote by $\hat{G}_{\ql}$. We further fix a pinning $(\hat{G},\hat{B},\hat{T},\hat{X})$ of $\hat{G}$. Choose an isomorphism $\iota:\mathbb{C}\simeq \bar{\mathbb{Q}}_p$. Let
$\{\mu_i\}$ denote the conjugacy class of Hodge cocharacters determined by $X_i$.

Let $V_i:=V_{\mu_i}$ denote the irreducible representation of $\hat{G}_{\ql}$ of highest weight $\mu_i$. Let $\nu\mid p$ be a place of the compositum of reflex field $E_1$,$E_2$ of $(G_1,X_1)$, $(G_2,X_2)$ (determined
by our choice of isomorphism $\iota$). Write $k_{\nu}$ for the residue field of $E_1E_2$ at $\nu$. Results of Kisin \cite{kisin2010integral} and Vasiu  \cite{vasiu2007good} state that there exists a canonical smooth integral model of $\textnormal{Sh}_K(G_i,X_i)$ over $\mathcal{O}_{E,(\nu)}$. Let  $d_i=\dim \textnormal{Sh}_K(G_i,X_i)$ and $\textnormal{Sh}_{\mu_i}$ denote the mod $p$ fiber of this canonical integral model, base changed to $\bar{k}_\nu$. Our assumption on $p$ implies that the action of the Galois group $\textnormal{Gal}(\Bar{\mathbb{Q}}_p/\mathbb{Q}_p)$ on $(\hat{G},\hat{B},\hat{T},\hat{X})$ factors through some finite quotient $\textnormal{Gal}(\mathbb{F}_{p^n}/\mathbb{F}_p)$ for some finite field $\mathbb{F}_{p^n}$ which contains $k_{\nu}$. Write $\sigma\in \textnormal{Gal}(\mathbb{F}_{p^n}/\mathbb{F}_p)$ for the arithmetic Frobenius. Consider the conjugation action of $\hat{G}_{\ql}$ on the (non-neutral) component $\hat{G}_{\ql}\sigma\subset \hat{G}_{\ql}\rtimes \langle\sigma\rangle$. Denote by $\textnormal{Coh}^{\hat{G}}(\hat{G}\sigma)$ the abelian category of coherent sheaves on the quotient stack $[\hat{G}_{\ql}\sigma/\hat{G}_{\ql}]$.

To each representation $W$ of $G_{\ql}$, we can attach an $\ell$-adic \'etale local system $\mathcal{L}_{i,W,\ql}$ on $\textnormal{Sh}_{\mu_i}$ by varying the level structure at $\ell$ (see \S$6.2$). The natural projection $[\hat{G}_{\ql}\sigma/\hat{G}_{\ql}]\rightarrow \mathbb{B}\hat{G}_{\ql}$ attaches to each representation $V$ of $\hat{G}_{\ql}$ a vector bundle $\widetilde{V}$ on $[\hat{G}_{\ql}\sigma/\hat{G}_{\ql}]$. Denote the global section of the structure sheaf on the quotient stack $[\hat{G}\sigma/\hat{G}]$ by $\mathcal{J}$, and the prime-to-$p$ Hecke algebra by $\mathcal{H}^p$. Fix a half Tate twist $\ql(1/2)$.

We state our main theorem.
\begin{theorem}
Let $\mathcal{L}_i:=\mathcal{L}_{i,W,\ql}[d_i](d_i/2)$. Under a  mild assumption, 
\begin{itemize}
\item[(1)] there exists a map 
\begin{equation}
  \textnormal{Spc}:\textnormal{Hom}_{\textnormal{Coh}^{\hat{G}}(\hat{G}\sigma)}(\widetilde{V_1},\widetilde{V_2})\rightarrow  \textnormal{Hom}_{\mathcal{H}^p\otimes \mathcal{J}}(\textnormal{H}_c^{*}(\textnormal{Sh}_{\mu_1},\mathcal{L}_{1}), \textnormal{H}_c^{*}(\textnormal{Sh}_{\mu_2},\mathcal{L}_{2})),
\end{equation}
which is compatible with compositions in the source and target. 
\item[(2)] the ring of endomorphisms $\textnormal{End}_{[\hat{G}\sigma/\hat{G}]}(\mathcal{O}_{[\hat{G}\sigma/\hat{G}]})$ acts on the compactly supported cohomology $\textnormal{H}_c^{*}(\textnormal{Sh}_{\mu_i},\mathcal{L}_i)$ via $\textnormal{Spc}$ and this action can be identified with the classical Satake isomorphism if $\textnormal{Sh}_{K}(G_1,X_1)=\textnormal{Sh}_{K}(G_2,X_2)$ is a Shimura set. 
\end{itemize}
\end{theorem}

This is a Jacquet-Langlands transfer for automorphic forms of higher weights which generalizes a previous construction given in \cite{xiao2017cycles}. We briefly discuss the proof of Theorem $1.1$.
\subsubsection{Strategy of Proof}
Let $\Lambda=\mathbb{Z}_\ell,\mathbb{F}_{\ell}$. The following theorem plays an essential role in the proof of the main result.
\begin{theorem}
For any projective $\Lambda$-modules $\Lambda_1,\Lambda_2\in\textnormal{Rep}_{\Lambda}(\hat{G}_{\Lambda})$, we choose appropriate integers $(m_1,n_1,m_2,n_2)$ and a dominant coweight $\lambda$ of $G$, and consider the \textit{Hecke correspondence} of moduli of local Shtukas
\begin{equation}
\begin{tikzcd}
\sht^{\loc(m_1,n_1)}_{\Lambda_1} \arrow[r,leftarrow,"h_{\Lambda_1}^{\leftarrow}"] &  \sht^{\lambda,\loc(m_1,n_1)}_{\Lambda_1\mid \Lambda_2} \arrow[r,"h_{\Lambda_2}^{\rightarrow}"] & 
\sht^{\loc(m_2,n_2)}_{\Lambda_2}
\end{tikzcd}.
\end{equation}
Then  there exists the following map
\begin{equation}
    \mathcal{S}_{\Lambda_1,\Lambda_2}:\textnormal{Hom}_{\textnormal{Coh}^{\hat{G}_{\Lambda}}(\hat{G}_{\Lambda}\sigma)}(\widetilde{\Lambda_1},\widetilde{\Lambda_2})\longrightarrow \textnormal{Hom}_{\textnormal{D}(\sht^{\loc(m_1,n_1)}_{\Lambda_1\mid \Lambda_2})}\Big ((h_{\Lambda_1}^{\leftarrow})^* S(\widetilde{\Lambda_1}),(h_{\Lambda_2}^{\rightarrow})^! S(\widetilde{\Lambda_2}) \Big ),
\end{equation}
which is independent of auxiliary choices.
\end{theorem}

We remark that the target of $\mathcal{S}_{\Lambda_1,\Lambda_2}$ can be understood as limits of the spaces of \textit{cohomological correspondences} between  $(\sht^{\loc(m_1,n_1)}_{\Lambda_1},S(\widetilde{\Lambda_1}))$ and $(\sht^{\loc(m_2,n_2)}_{\Lambda_2},S(\widetilde{\Lambda_2}))$ which are supported on the Hecke correspondence $(1.2)$. In \cite{xiao2017cycles}, Xiao-Zhu construct the maps $\sv$ for $\bar{\mathbb{Q}}_{\ell}$-representations $\hat{G}_{\bar{\mathbb{Q}}_\ell}$ in a categorical way. In fact, they define the category $\textnormal{P}^{\hk}(\sht^{\loc},\bar{\mathbb{Q}}_{\ell})$. Its objects are perverse sheaves supported on finite dimensional subspace of $\textnormal{Sht}^{\textnormal{loc}}_{\bar{\mathbb{F}}_q}$, and its spaces of morphisms are given by (limits of) cohomological correspondences supported on Hecke correspondences of \textit{restricted} local Shtukas. It receives a canonical functor from the 'normal' category of perverse sheaves $\textnormal{P}(\sht^{\loc},\bar{\mathbb{Q}}_{\ell})$, and this functor has a $\sigma$-twisted trace structure. Hence, the universal property of categorical traces asserts that it admits a functor from the $\sigma$-twisted categorical trace  $\textnormal{Tr}_{\sigma}(\textnormal{Rep}(\hat{G}_{\bar{\mathbb{Q}}_{\ell}}))\cong \textnormal{Coh}^{\hat{G}_{\bar{\mathbb{Q}}_{\ell}}}_{\textnormal{fr}}(\hat{G}_{\bar{\mathbb{Q}}_{\ell}}\sigma)$ in the sense of \cite{zhu2018geometric}. Here, the latter category is the full subcategory of $\textnormal{Coh}^{\hat{G}_{\bar{\mathbb{Q}}_{\ell}}}(\hat{G}_{\bar{\mathbb{Q}}_{\ell}}\sigma)$ generated by objects coming from $\textnormal{Rep}_{\bar{\mathbb{Q}}_\ell}(\hat{G}_{\bar{\mathbb{Q}}_{\ell}})$. This idea is made more explicit in \textit{loc.cit}.

The above strategy does not carry over in our situation. First of all, it is not clear how to a suitable notion of categorical trace of To calculate the left $\sigma$-twisted categorical trace of $\textnormal{Rep}_{\mathbb{F}_{\ell}}(\hat{G}_{\mathbb{F}_{\ell}})$. To accomplish this, we need to appeal to a more general construction of the tensor product for finitely cocomplete categories. In addition, the correspondence category $\textnormal{P}^{\cc}(\sht^{\loc},\mathbb{Z}_{\ell})$ is not finitely cocomplete any more and the desired maps $\mathcal{S}$ therefore cannot be obtained by the universal property of the categorical trace. One possible way to overcome this difficulty is to upgrade $\textnormal{P}^{\cc}(\sht^{\loc},\mathbb{Z}_{\ell})$ to a higher category.

Instead of pursuing this idea, we take a more concrete approach. We note that there is a natural isomorphism 
$$
\textnormal{Hom}_{\textnormal{Coh}^{\hat{G}_{\Lambda}}(\hat{G}_{\Lambda}\sigma)}(\widetilde{\Lambda_1},\widetilde{\Lambda_2})\cong \textnormal{Hom}_{\hat{G}_{\Lambda}}(\Lambda_1,\mathcal{O}_G\otimes \Lambda_2),
$$
where $\mathcal{O}_G$ denotes the regular representation of $\hat{G}_{\Lambda}$. By the Peter-Weyl theorem, $\mathcal{O}_G$ admits a filtration with associated graded $\oplus_{\lambda}W(\lambda)\otimes S(\lambda^{*})$ where $W(\lambda)$ denotes the Schur module of $\hat{G}_{\Lambda}$. For any $W\in\textnormal{Rep}(\hat{G})$, and any $\mathbf{a}\in \hg(V,W\otimes S(\lambda^*)\otimes W)$, we use the integral coefficient geometric Satake equivalence discussed in \S 2.4 to construct a cohomological correspondence on restricted local Hecke stacks of $G\times G$. The maps $\mathcal{S}_{V,W}$ are constructed by first pulling this cohomological correspondence back to a cohomological correspondence on restricted local Hecke stacks and then pulling it back to  a cohomological correspondence on restricted local Shtukas.

\subsection{Plan of the paper}
We briefly discuss the organization of the paper. In \S 2, we recall the geometric properties of mixed characteristic affine Grassmannians and the integral coefficient geometric Satake equivalence following \cite{zhu2017affine} and \cite{yu2019integral}. In \S 3 and \S 4, we summarize the geometry of local Hecke stacks and moduli of local Shtukas. In addition, we define categories of perverse sheaves on these two stacks and study the cohomological correspondences in these categories. Results appeared in these two sections are proved in \cite{xiao2017cycles}, and we refer to \textit{loc.cit} for detailed proofs. Section 5 is devoted to study the cohomological correspondences between perverse sheaves on the stacks of restricted local Shtukas which come from the integral geometric Satake equivalence. We derive the key ingredient of the main theorem. We apply the results in previous sections to study the cohomological correspondences between mod $p$ fibres of different Hodge type Shimura varieties and prove our main theorem in \S6.

\subsection{Notations}
We fix notations for later use.

Let $F$ be a mixed characteristic local ring with ring of integers $\mathcal{O}$ and residue field $k=\mathbb{F}_q$.  We write $\sigma$ for the arithmetic Frobenius of $\mathbb{F}_q$. For any $k$-algebra $R$, its ring of Witt vectors is denoted by 
$$
W(R)=\{(r_0,r_1,\cdots)\mid r_i\in R\}.
$$
We denote by $W_h(R)$ the ring of truncated Witt vectors of length $h$. Write $\textbf{Aff}^{\textnormal{pf}}_k$ for the category of perfect $k$-algebras. For any $R\in \textbf{Aff}^{\textnormal{pf}}_k$,  we know that $W_h(R) = W(R)/p^hW(R)$. We define the ring of Witt vectors in $R$ with coefficient in $\mathcal{O}$ as
$$
W_{\mathcal{O}}(R):=W(R)\hat{\otimes}_{W(k)} \mathcal{O}:=\varprojlim_nW_{\mathcal{O},n}(R),\text{ and }W_{\mathcal{O},n}(R)=W(R)\otimes_{W(k)} \mathcal{O}/\varpi^n.
$$
We define the $n$-th unit disk, formal unit disk, and formal punctured unit disk to be
$$
D_{n,R}:=\textnormal{Spec}W_{\mathcal{O},n}(R),\textnormal{  }D_{R}:=\textnormal{Spec}W_{\mathcal{O}}(R),\text{ } D_{R}^{\times}:=\textnormal{Spec}W_{\mathcal{O}}(R)[1/\varpi],
$$
respectively.

Let $L$ be the completion of the maximal unramified extension of $F$. Denote by $\mathcal{O}_L$ its ring of integers, and we fix a uniformizer $\varpi\in\mathcal{O}_L$. 
We will assume $G$ to be an unramified reductive group scheme over $\mathcal{O}$. We denote by $T$ the abstract Cartan subgroup of $G$. Let $S \subset T$ denote the maximal split subtorus. In the case where $G$ is a split reductive group, we will choose a Borel subgroup $B\subset G$ over $\mathcal{O}$ and a split maximal torus $T\subset B$. When we need to embed $T$ (or $S$) into $G$ as a (split) maximal torus, we will state it explicitly.

 Let $\mathbb{X}_{\bullet}$ denote the coweight lattice of $T$ and $\mathbb{X}^{\bullet}$ the weight lattice. Let $\Delta\subset\mathbb{X}^{\bullet}$ (resp. $\Delta^{\vee}\subset\mathbb{X}_{\bullet}$) the set of roots  (resp. coroots). A choice of the Borel subgroup $B\subset G$ determines the semi-group of dominant coweights $\mathbb{X}_{\bullet}^{+}\subset \mathbb{X}_{\bullet}$ and the set of positive roots $\Delta_+\subset \Delta$. In fact, $\mathbb{X}_{\bullet}^+$ and $\Delta_+$ are both independent of the choice of $B$. The
$q$-power (arithmetic) Frobenius $\sigma$ acts on $(\mathbb{X}^{\bullet},\Delta,\mathbb{X}_{\bullet},\Delta^{\vee})$ preserving $\mathbb{X}_{\bullet}^+$. The Langlands dual group of $G$ is denoted by $\hat{G}$.

Let $2\rho\in \mathbb{X}^{\bullet}$ be the sum of all positive roots. Define the partial order $``\leq"$ on $\mathbb{X}_{\bullet}$ to be such that $\lambda\leq \mu$ if and only if  $\mu-\lambda$  equals a non-negative integral linear combination of positive coroots.
For any $\mu\in\mathbb{X}_{\bullet}$, denote $\varpi^{\mu}$ by the image of $\mu$ under the composition of maps
$$
\mathbb{G}_m\rightarrow T\subset G.
$$

Let $H$ be a reductive group over a field $K$. We write $H_{\textnormal{ad}}$ for its adjoint group, $H_{\textnormal{der}}$ for its derived group, and $H_{\textnormal{sc}}$ for the simply connected cover of $H_{\textnormal{ad}}$. 
Let $T\subset H$ be a maximal torus, we denote by $T_{\textnormal{ad}}$ its image in the quotient $G_{\textnormal{ad}}$. We also write $T_{\textnormal{sc}}$ for the preimage of $T$ in $G_{\textnormal{sc}}$.

We denote by $\mathcal{E}^0$ the trivial $G$-torsor. For any perfect $k$-algebra $R$, the arithmetic Frobenius $\sigma$ induces an automorphism $\sigma\otimes \textnormal{id}$ of the $D_R$. Let $\mathcal{E}$ be a $G$-torsor over $D_R$, and we denote the $G$-torsor $(\sigma\otimes \textnormal{id})^*\mathcal{E}$ by  $^{\sigma}\mathcal{E}$.

Let $X$ be an algebraic space over $k$. We write $\sigma_X$ for the absolute Frobenius morphism of $X$. We denote by $X^{p^{\-\infty}}:=\varprojlim_{\sigma_X}X$ the perfection of $X$.

Let $\ell\neq p$ be a prime number. Fix a half Tate twist $\mathbb{Z}_{\ell}(1/2)$. We write $\langle  d \rangle:=[d](d/2)$. Let $X$ and $Y$ be two algebraic stacks which are perfectly of finitely presentation in the opposite category of perfect $k$-algebras. For each $f:X\rightarrow Y$ being a perfectly smooth morphism of relative dimension $d$, we write $f^{\star}:=f^*\langle  d \rangle$.

Throughout the paper, we write $\Lambda$ for $\mathbb{Z}_{\ell}$ and $\mathbb{F}_\ell$ and $E$ for $\mathbb{Z}_{\ell}$ unless otherwise stated. For stacks $X_1$, $X_2$, and perverse sheaves $\mathcal{F}_i\in \textnormal{P}(X_i,\Lambda)$, we sometimes write the space of cohomological correspondences $\cc_X((X_1,\mathcal{F}_1),(X_2,\mathcal{F}_2))$ as $\cc_X(\mathcal{F}_1,\mathcal{F}_2)$ for simplicity.

Let $X$ be a stack, we denote by $\omega_X\in D_b^c(X,E)$ the dualizing sheaf of $X$  in the bounded derived category of sheaves on $X$. For a perfect pfp algebraic space (cf.\cite[A.1.7]{xiao2017cycles}) $X$ over $k$, we denote by $\textnormal{H}^{BM}_i(X_{\bar{k}}):=\textnormal{H}^{-i}(X_{\bar{k}},\omega_x(-i/2))$ the $i$th Borel-Moore homology of $X_{\bar{k}}$.

\subsection{Acknowledgements}
I am very grateful to my advisor Xinwen Zhu for introducing me to the beautiful theory of moduli of Shtukas to me and providing  numerous helpful comments and discussions while I was working on this project.

\section{The Integral Coefficient Geometric Satake Equivalence in Mixed Characteristic}
We briefly recall the geometric properties of the mixed characteristic affine Grassmannians and the integral coefficient geometric Satake equivalence. We refer to \cite{zhu2017affine} and \cite{yu2019integral} for more details. 
\subsection{The geometry of mixed characteristic affine Grassmannians}
\begin{definition}
The mixed characteristic \textit{affine Grassmannian} $\textnormal{Gr}_G$ of $G$ over $k$ is defined as the presheaf over $\textbf{Aff}^{\textnormal{pf}_k}$
\begin{equation*}
\textnormal{Gr}_G(R)=\left\{(\mathcal{E},\phi)\bigg|
\begin{aligned}
&	\mathcal{E}\rightarrow D_{R} \text{ is a } G\text{-torsor, and } \\
& \phi:\mathcal{E}\mid_{D_{R}^{\times}}\simeq \mathcal{E}^{0}\mid_{D_{R}^{\times}} \\
\end{aligned}
\right\}.
\end{equation*}
\end{definition}
In addition to this moduli interpretation, it can be understood as a homogeneous space. We define the $p$-adic \textit{jet group} (resp.\textit{loop group}) as the presheaf  $L^+G(R):=G(W_{\mathcal{O}}(R))$ (resp. $LG:=G(W_{\mathcal{O}}(R[1/\varpi])$) over the $\textbf{Aff}^{\textnormal{pf}}_k$. We define $L^nG(R):=G(W_{\mathcal{O},n}(R))$ to be the $n$-\textit{th jet group}. It is represented by the perfection of of an algebraic $k$-group and we have $L^+G=\varprojlim_n L^nG$. We define the $n$-\textit{th congruence group} $L^+G^{(n)}(R):=\textnormal{ker}(L^+G\rightarrow L^nG(R))$.
Then $\textnormal{Gr}_G$ can be realized as the fpqc quotient space $[LG/L^+G]$.  It is clear that $LG\rightarrow \textnormal{Gr}_G$ is an $L^{+}G$-torsor and $L^{+}G$ naturally acts on $\textnormal{Gr}_G$. We can form the twisted product which we also call the \textit{convolution product} of affine Grassmannians  
$$
\textnormal{Gr}_G\tilde{\times}\textnormal{Gr}_{G}:=LG\times^{L^{+}G}\textnormal{Gr}_G:=[LG\times \textnormal{Gr}_G/L^{+}G],
$$
where $L^{+}G$ acts on $LG\times \textnormal{Gr}_G$ anti-diagonally as $g^{+}\cdot ([g_1],[g_2]):=([g_1(g^{+})^{-1}],[g^{+}g_2])$. It also admits a moduli interpretation as follows
\begin{equation*}
\textnormal{Gr}_G\tilde{\times}\textnormal{Gr}_{G}(R)=\left\{(\mathcal{E}_1,\mathcal{E}_2,\beta_1,\beta_2)\bigg|
\begin{aligned}
&	\mathcal{E}_1,\mathcal{E}_2 \text{ are } G-\text{torsors on }D_{R}, \text{and} \\
& \beta_1:\mathcal{E}_1\mid_{D_{R}^{\times}}\simeq \mathcal{E}^0\mid_{D_{R}^{\times}}, \beta_2:\mathcal{E}_2\mid_{D_{R}^{\times}}\simeq \mathcal{E}_1\mid_{D_{R}^{\times}} \\
\end{aligned}
\right\}.
\end{equation*}
We define the convolution morphism
$$
m:\textnormal{Gr}_G\tilde{\times}\textnormal{Gr}_{G}\longrightarrow \textnormal{Gr}_G, 
$$
such that 
$$
(\mathcal{E}_1,\mathcal{E}_2,\beta_1,\beta_2)\longmapsto (\mathcal{E}_2,\beta_1\beta_2).
$$
\\
The celebrated theorem of Bhatt-Scholze \cite{bhatt2017projectivity} shows that the functor $\textnormal{Gr}_{G}$ is representable by an inductive limit of perfections
of projective varieties.

For each algebraically closed field $K$ containing $\bar{k}$, the moduli interpretation of affine Grassmannians assigns to each $(\mathcal{E},\phi)\in \textnormal{Gr}_G(K)$ an element in $G(W_{\mathcal{O}}(K))\backslash G(\mathcal{O}(K)[1/\varpi])/G(\mathcal{O}(K))\simeq \mathbb{X}_{\bullet}^{+}$ by the Cartan decomposition. We denote the corresponding dominant coweight by $\textnormal{Inv}(\phi)$ and call it the \textit{relative position} of $(\mathcal{E},\phi)$. 
\begin{definition}
For each $\mu\in \mathbb{X}_{\bullet}^{+}$, we define 
\begin{itemize}
    \item [(1)] the (spherical) \textit{Schubert variety}
    $$
\textnormal{Gr}_{\mu}:=\{(\mathcal{E},\phi)\in \textnormal{Gr}_G\mid \textnormal{Inv}(\phi)\leq \mu\},
$$ 
    \item[(2)] the \textit{Schubert cell}
    $$
\mathring{\textnormal{Gr}_{\mu}}:=\{(\mathcal{E},\phi)\in \textnormal{Gr}_G\mid \textnormal{Inv}(\phi)= \mu\}.
$$ 
\end{itemize} 

\end{definition}
The Schubert cell $\mathring{\textnormal{Gr}_{\mu}}$ is the orbit of $\varpi^{\mu}$ in $\textnormal{Gr}_G$ under the $L^+G$-action. It
is the perfection of a quasi-projective smooth variety of dimension $(2\rho,\mu)$, and $L^+G$ acts on it via a finite type quotient $L^mG$. Taking the Zariski closure of $\mathring{\textnormal{Gr}_{\mu}}$ gives $\textnormal{Gr}_{\mu}$. Moreover, for any $\mu,\nu\in\mathbb{X}^+_{\bullet}$, $\mu\leq\nu$ if and only if $\textnormal{Gr}_{\mu}\subseteq \textnormal{Gr}_{\nu}$.

\subsection{The Satake category}
We know that $\pi_0(\textnormal{Gr}_G\otimes\bar{k})\simeq \pi_1(G)$ (cf. \cite[Proposition 1.21]{zhu2017affine}). The affine Grassmannian $\textnormal{Gr}_G\otimes \bar{k}$ has the decomposition into connected components
$$
\textnormal{Gr}_G\otimes \bar{k}=\sqcup_{\xi\in\pi_1(G)}\textnormal{Gr}_{\xi}.
$$
Recall our discussion in \S $2.1$. We have 
$$
\textnormal{Gr}_{\xi}=\varinjlim_{\mu\in \xi}\textnormal{Gr}_{\mu},
$$
where $\mu\in \xi$ means that the natural map $\mathbb{X}_{\bullet}\rightarrow \pi_1(G)$ sends $\mu$ to $\xi$. The connecting morphism $i_{\mu,\nu}:\textnormal{Gr}_{\mu}\rightarrow \textnormal{Gr}_{\nu}$ is a closed embedding if $\mu\leq \nu$. For $m\leq m'$ be integers such that $L^+G$ acts on $\textnormal{Gr}_{\mu}$ through $L^mG$ and $L^{m'}G$, there is a canonical equivalence
$$
P_{L^mG\otimes\bar{k}}(\textnormal{Gr}_{\mu},\Lambda)\cong P_{L^{m'}G\otimes\bar{k}}(\textnormal{Gr}_{\mu},\Lambda).
$$
We define the category of $L^+G\otimes \bar{k}$-equivariant $\Lambda$-coefficient perverse sheaves on $\textnormal{Gr}_G\otimes \bar{k}$ as
\begin{align*}
\textnormal{P}_{L^+G\otimes\bar{k}}(\textnormal{Gr}_G\otimes\bar{k},\Lambda):= \bigoplus_{\xi\in\pi_1(G)}\textnormal{P}_{L^+G\otimes\bar{k}}(\textnormal{Gr}_{\xi},\Lambda),\\
\textnormal{P}_{L^+G\otimes\bar{k}}(\textnormal{Gr}_{\xi},\Lambda):=\varinjlim_{(\mu,m)}P_{L^mG\otimes \bar{k}}(\textnormal{Gr}_{\mu},\Lambda).
\end{align*}
Here, the limit is taken over the pairs $\{(\mu,m)\mid\mu\in\xi,m\textnormal{ is large enough} \}$ with partial order given by $(\mu,m)\leq (\mu',m)$ if $\mu\leq\mu'$ and $m\leq m'$. The connecting morphism is the composition
$$
\textnormal{P}_{L^mG\otimes\bar{k}}(\textnormal{Gr}_{\mu},\Lambda)\cong \textnormal{P}_{L^{m'}G\otimes\bar{k}}(\textnormal{Gr}_{\mu},\Lambda)\xrightarrow[]{i_{\mu,\mu'}^*}\textnormal{P}_{L^{m'}G\otimes\bar{k}}(\textnormal{Gr}_{\mu'},\Lambda)
$$
which is a fully faithful embedding.
We also call this category the \textit{Satake category} and sometimes denote it by $\textnormal{Sat}_{G,\Lambda}$ for simplicity. We denote by $\textnormal{IC}_{\mu}$ for each $\mu\in \mathbb{X}_{\bullet}^{+}$ the \textit{intersection cohomology sheaf} on $\textnormal{Gr}_{\leq\mu}$. Its restriction to each open strata $\textnormal{Gr}_{\mu}$ is constant and in particular, $\textnormal{IC}_{\mu}\mid_{\textnormal{Gr}_{\mu}}\simeq  \Lambda[(2\rho,\mu)]$.

We recall the the monoidal structure in $\textnormal{Sat}_{G,\Lambda}$ defined by Lusztig's convolution of sheaves.
Consider the following diagram 
$$
\textnormal{Gr}_G\times \textnormal{Gr}_G\stackrel{p}{\longleftarrow}LG\times \textnormal{Gr}_G\stackrel{q}{\longrightarrow} \textnormal{Gr}_G\tilde{\times} \textnormal{Gr}_G\stackrel{m}{\longrightarrow} \textnormal{Gr}_G,
$$
where $p$ and $q$ are natural projections. We define for any $\mathcal{A}_1,\mathcal{A}_2\in P_{L^+G}(\textnormal{Gr}_G,\Lambda)$, 
$$\mathcal{A}\star \mathcal{A}_2:=Rm_!(\mathcal{A}_1\tilde{\boxtimes}\mathcal{A}_2),$$
where $\mathcal{A}_1\tilde{\boxtimes}\mathcal{A}_2\in P_{L^+G}(\textnormal{Gr}_G\tilde{\times} \textnormal{Gr}_G,\Lambda)$ is the unique object such that 
$$
q^*(\mathcal{A}_1\tilde{\boxtimes}\mathcal{A}_2)\simeq p^*({^{p}\textnormal{H}^{0}}(\mathcal{A}_1\boxtimes\mathcal{A}_2)).
$$ 
We remark that unlike the construction in $P_{L^+G}(\textnormal{Gr}_G,\bar{\mathbb{Q}}_{\ell})$, taking the $0$-th perverse cohomology in the above definition is necessary. This is because when we work with $\mathbb{Z}_{\ell}$-sheaves, the sheaf $\mathcal{A}_1\boxtimes\mathcal{A}_2$ may not be perverse. In fact, $\mathcal{A}_1\boxtimes\mathcal{A}_2$ is perverse if one of $\textnormal{H}^{*}(\mathcal{A}_i)$ is a flat $\mathbb{Z}_{\ell}$-module. For more details, please refer to \cite[Lemma 4.1]{mirkovic2007geometric}. A priori, the convolution product $\mathcal{A}_1\star\mathcal{A}_2$ is a complex of of $L^+G$ equivariant sheaves on $\textnormal{Gr}_G$. It is the semi-smallness of the convolution morphism $m$ (cf.\cite{zhu2017affine}) that guarantees $\mathcal{F}_1\star \mathcal{F}_2$ to be a perverse sheaf. 

\subsection{More on the convolution Grassmannians}
We recall the definition and geometry of the convolution product of multiple copies of affine Grassmannians and related constructions. A detailed exposition maybe found in \cite[\S3.1]{zhu2017affine}.

Let $\mb=(\mu_1,\mu_2,\cdots,\mu_n)$ be a sequence of dominant coweights.  
\begin{definition}
We define the convolution product of $\textnormal{Gr}_{\mu_1},\textnormal{Gr}_{\mu_2},\cdots, \textnormal{Gr}_{\mu_n}$ as the presheaf $\textnormal{Gr}_{\mb}$ which classifies isomorphism classes of modifications of $G$-torsors over $D_{R}$ 
\begin{equation}
\begin{tikzcd}[sep=small]
\mathcal{E}_n \arrow[r,dashed,"\beta_1"] & \mathcal{E}_{n-1} \arrow[r,dashrightarrow,"\beta_{n-1}"]& \cdots \arrow[r,dashrightarrow,"\beta_1"] & \mathcal{E}_0=\mathcal{E}^0,
\end{tikzcd}
\end{equation}
where $\textnormal{Inv}(\beta_i)\leq \mu_i$.
\end{definition}
We define $\textnormal{Gr}_{\leq\mb}^{(\infty)}$ as the $L^+G$-torsor over $\textnormal{Gr}_{\leq\mb}$ which classifies a point in $\textnormal{Gr}_{\leq\mb}$ as $(2.1)$ together with an isomorphism $\mathcal{E}_n\simeq \mathcal{E}^0$. For any integer $n$, we define $\textnormal{Gr}_{\leq\mb}^{(n)}$ to be the $L^nG$-torsor over $\textnormal{Gr}_{\leq\mb}$ which classifies a point in $\textnormal{Gr}_{\leq\mb}$ together with an isomorphism $\mathcal{E}_n\mid_{D_{n,R}}\simeq \mathcal{E}^{0}\mid_{D_{n,R}}$. For any $m<n$, there is an isomorphism 
$$
\textnormal{Gr}_{\leq\mb}\simeq \textnormal{Gr}_{\leq\mu_1,\cdots,\mu_m}\tilde{\times} \textnormal{Gr}_{\leq\mu_{m+1},\cdots,\mu_n}:= \textnormal{Gr}_{\leq\mu_1,\cdots,\mu_m}^{(\infty)}\times ^{L^+G}\textnormal{Gr}_{\leq\mu_{m+1},\cdots,\mu_n}.
$$

Many of the constructions in later sections make use of the following lemma \cite[Lemma 3.1.7]{zhu2017affine}:
\begin{lemma}
For any sequence of dominant coweights $\mb=(\mu_1,\mu_2,\cdots,\mu_n)$, there exists a non-negative integer $m$, such that for any non-negative integer $n$, the action of $L^{m+n}G$ on $\textnormal{Gr}_{\leq\mb}$ is trivial. In other words, the natural action of $L^+G$ on $\textnormal{Gr}_{\leq\mb}$ factors through the finite type quotient $L^{m+n}G$.
\end{lemma}
We will call such an integer $m$ a $\mb$-\textit{large} integer. We also call a pair of non-negative integers $(m,n)$ ($m=\infty$ allowed) to be $\mb$-\textit{large} if $m-n$ is a $\mb$-large integer.

Replacing $\textnormal{Gr}_{\leq\mu_i}$ by $\textnormal{Gr}_{\mu_i}$, we can similarly define $\textnormal{Gr}_{\mb}:=\textnormal{Gr}_{\mu_1}\tilde{\times}\cdots\tilde{\times} \textnormal{Gr}_{\mu_n}$. It admits a stratification
\begin{equation}
\textnormal{Gr}_{\leq\mb}=\cup_{\mb'\leq \mb}\textnormal{Gr}_{\mb'},
\end{equation}
where $\mb'\leq \mb$ means $\mu_i'\leq\mu_i$ for each $i$. 

As in \cite{zhu2017affine}, we let $|\mb|:=\sum\mu_i$. Then the convolution map induces the following morphism
$$
m:\textnormal{Gr}_{\leq\mb}\longrightarrow \textnormal{Gr}_{\leq |\mb|},
$$
such that
$$
(\mathcal{E}_i,\beta_i)\longmapsto (\mathcal{E}_n,\beta_1\cdots\beta_n).
$$
Let $\nb$ be another sequence of dominant coweights. We define the following stack 
$$
\textnormal{Gr}_{\mb\mid\nb}^{0}:=\textnormal{Gr}_{\leq\mb}\times_{m_{\mb},\textnormal{Gr}_G,m_{\nb}} \textnormal{Gr}_{\leq\nb}.
$$
Write the natural projections from $\textnormal{Gr}_{\mb\mid\nb}^{0}$ to $\textnormal{Gr}_{\leq\mb}$ and $\textnormal{Gr}_{\leq\nb}$ as $h^{\leftarrow}_{\mb}$ and $h^{\rightarrow}_{\nb}$, respectively. We call the following diagram 
\begin{equation}
\begin{tikzcd}
\textnormal{Gr}_{\leq\mb} \arrow[r,leftarrow, "h^{\leftarrow}_{\mb}"] & \Gmbnb^0  \arrow[r,"h^{\rightarrow}_{\nb}"] & \textnormal{Gr}_{\leq\nb}
\end{tikzcd}
\end{equation}
the \textit{Satake correspondence}.
\begin{definition}
An irreducible component of $\Gmbnb^0$
of dimension $(\rho,\vert\mb\vert+\vert\nb\vert)$  is called a \textit{Satake cycle}. Denote the set of Satake cycles of $\Gmbnb^0$ by $\mathbb{S}_{\mb\mid\nb}$. For $\mathbf{a}\in \mathbb{S}_{\mb\mid\nb}$, write $\Gmbnb^{0,\mathbf{a}}$ for the Satake cycle labelled by $\mathbf{a}$.
\end{definition}
It is clear that $\Gmbnb^0\cong \textnormal{Gr}_{\nb\mid\mb}^0$, and we conclude that 
\begin{equation}
\mathbb{S}_{\mb\mid\nb}=\mathbb{S}_{\nb\mid\mb}.
\end{equation}

\subsection{The geometric Satake equivalence}
We state the geometric Satake equivalence \cite{yu2019integral} which is an indispensable ingredient of our construction of the Jacquet-Langlands correspondence. 
\begin{theorem}
There is an equivalence of monoidal categories 
\begin{equation}
    \textnormal{Sat}_G:\ \textnormal{Rep}_{\Lambda}(\hat{G}_{\Lambda})\simeq \textnormal{Sat}_{G,\Lambda},
\end{equation}
where $\textnormal{Rep}_{\Lambda}(\hat{G}_{\Lambda})$ denotes the category of $\Lambda$-representations of $\hat{G}_{\Lambda}$ which are finitely generated over $\Lambda$. In particular, for any sequence $\mb$ of dominant coweights, the Satake equivalence assigns $V_{\mb}$ the tensor product of highest weight representations $V_{\mu_i}$ corresponding to $\mu_i$ to the convolution product $\textnormal{IC}_{\mb}$ of intersection cohomology sheaf $\textnormal{IC}_{\mu_i}$.
\end{theorem}

\begin{remark}
As explained in \cite[\S 5.5]{zhu2016introduction}, the Galois group $\textnormal{Gal}(\bar{\mathbb{F}}_p/\mathbb{F}_p)$ acts on the Satake category $
\textnormal{Sat}_{G,\Lambda}$ by tensor auto-equivalences. It in turn induces an action of $\textnormal{Gal}(\bar{\mathbb{F}}_p/\mathbb{F}_p)$ on $\hat{G}$ which preserves $(\hat{G},\hat{B},\hat{T})$. Let $V\in\textnormal{Rep}_{E}(\hat{G})$ and $\gamma\in \textnormal{Gal}(\bar{\mathbb{F}}_p/\mathbb{F}_p)$. We write $^{\gamma}V$ for the representation 
$$
\hat{G}\xrightarrow{\gamma^{-1}} \hat{G}\rightarrow \textnormal{GL}_{E}(V)
$$
of $\hat{G}$.
\end{remark}

For three sequences of dominant weight $\mu_{1\bullet}$, $\mu_{2\bullet}$, and $\mu_{3\bullet}$ the following lemma is an immediate consequence of Theorem $2.6$.
\begin{corollary}
We have the following natural isomorphism 
\begin{equation}
    \hg(V_{\mu_{i\bullet}},V_{\mu_{j\bullet}})\cong \cc_{\textnormal{Gr}^0_{\mu_{i\bullet}\mid\mu_{j\bullet}}}((\textnormal{Gr}_{\leq\mu_{i\bullet}},\textnormal{IC}_{\mu_{i\bullet}}),(\textnormal{Gr}_{\leq\mu_{j\bullet}},\textnormal{IC}_{\mu_{j\bullet}})), \end{equation}
such that
the natural composition on the left hand 
$$
\hg(V_{\mu_{1\bullet}},V_{\mu_{2\bullet}})\otimes \hg(V_{\mu_{2\bullet}},V_{\mu_{3\bullet}})\rightarrow \hg(V_{\mu_{1\bullet}},V_{\mu_{3\bullet}})
$$
is compatible with the composition of cohomological correspondences on the right hand side
\begin{align*}
& \cc_{\textnormal{Gr}^0_{\mu_{1\bullet}\mid\mu_{2\bullet}}}((\textnormal{Gr}_{\leq\mu_{1\bullet}},\textnormal{IC}_{\mu_{1\bullet}}),(\textnormal{Gr}_{\leq\mu_{2\bullet}},\textnormal{IC}_{\mu_{2\bullet}}))\otimes \cc_{\textnormal{Gr}^0_{\mu_{2\bullet}\mid\mu_{3\bullet}}}((\textnormal{Gr}_{\leq\mu_{2\bullet}},\textnormal{IC}_{\mu_{2\bullet}}),(\textnormal{Gr}_{\leq\mu_{3\bullet}},\textnormal{IC}_{\mu_{3\bullet}}))\\
\rightarrow & \cc_{\textnormal{Gr}^0_{\mu_{1\bullet}\mid\mu_{3\bullet}}}((\textnormal{Gr}_{\leq\mu_{1\bullet}},\textnormal{IC}_{\mu_{1\bullet}}),(\textnormal{Gr}_{\leq\mu_{3\bullet}},\textnormal{IC}_{\mu_{3\bullet}}))
\end{align*}
which is obtained by pushing forward the cohomological correspondences along the natural map
$$
Gr^0_{\mu_{1\bullet}\mid\mu_{2\bullet}}\times_{\textnormal{Gr}_{\mu_{2\bullet}}} Gr^0_{\mu_{2\bullet}\mid\mu_{3\bullet}}\rightarrow Gr^0_{\mu_{1\bullet}\mid\mu_{3\bullet}}.
$$
In addition, there is a canonical isomorphism
\begin{equation}
    \textnormal{Hom}_{\textnormal{P}(\textnormal{Gr}_G)}(m_{\mb*}\textnormal{IC}_{\mb},m_{\nb*}\textnormal{IC}_{\nb})\cong \textnormal{H}^{\textnormal{BM}}_{(2\rho,\vert\mb\vert+\vert\nb\vert)}(\textnormal{Gr}_{\mb\mid \nb}^0).
\end{equation}
\end{corollary}
\begin{proof}
The lemma can be proved exactly as \cite[Corollary 3.4.4]{zhu2017affine}, and we refer to $\textit{loc.cit}$ for details of the proof.
\end{proof}

\section{Local Hecke Stacks}
We review the definition of local Hecke stacks and study their geometric properties following \cite{xiao2017cycles}.
\begin{definition}
Let $\mb=(\mu_1,\mu_2,\cdots,\mu_n)$ be a sequence of dominant coweights of $G$. The \textit{local Hecke stack} $\hk^{\loc}_{\mb}$ is defined as the moduli problem which assigns to each perfect $k$-algebra $R$ the groupoid of chains of modifications of $G$-torsors
\begin{equation}
\begin{tikzcd}[]
\mathcal{E}_n \arrow[r,dashed] & \mathcal{E}_{n-1} \arrow[r,dashrightarrow]& \cdots \arrow[r,dashrightarrow] & \mathcal{E}_0
\end{tikzcd}
\end{equation}
over $D_{R}$ of relative positions $\leq\mu_n,\cdots,\leq \mu_1$, respectively.
\end{definition}
It may also be understood as the homogeneous space $[L^+G\backslash \textnormal{Gr}_{\mu_{\leq\bullet}}]$. Similarly, we define $$\hk^{0,\loc}_{\mb\mid\nb}:=[L^+G\backslash \textnormal{Gr}_{\mb\mid\nb}^0]$$
as the stack which classifies for each perfect $k$-algebra $R$ the rectangles of modifications
$$
\begin{tikzcd}[]
\mathcal{E}_n\arrow[r,dashed] \arrow[d,equal] &  \cdots \arrow[r,dashrightarrow] & \mathcal{E}_0\arrow[d,equal] \\
 \mathcal{E}'_{m} \arrow[r,dashed]& \cdots \arrow[r,dashrightarrow] & \mathcal{E}'_0 ,
\end{tikzcd}
$$
of $G$-torsors over $D_R$ with modifications in the upper (resp. lower) row bounded by $\mb$ (resp. $\nb$).

Taking quotient of the Satake correspondence $(2.3)$ by $L^+G$, we get the \textit{Satake correspondence} for local Hecke stacks,
\begin{equation}
    \begin{tikzcd}
    \hk^{\loc}_{\mb}\arrow[r,leftarrow, "h_{\mb}^{\leftarrow}"] & \Hmbnb^{0,\loc} \arrow[r,"h_{\nb}^{\rightarrow}"] & \hk^{\loc}_{\nb}.
    \end{tikzcd}
\end{equation}
It is clear from the definition that these stacks are not of finite type, thus we need their finite dimensional quotient to apply the $\ell$-adic formalism.\begin{definition}
 For a sequence of dominant coweights $\mb=(\mu_1,\mu_2,\cdots,\mu_n)$, choose a $\mb$-large integer $m$, and we define  the $m$-\textit{restricted local Hecke stack} to be the stack
 $$
 \hk_{\mb}^{\loc(m)}:=[L^mG\backslash \textnormal{Gr}_{\leq\mu}].
 $$
\end{definition}
Similarly, choose $m$ large enough for $\mb$ and $\nb$, for example, $m$ is taken to be $(\mb,\nb)$-large, and we define $\hk_{\mb\mid\nb}^{0,\loc(m)}:=[L^mG\backslash \textnormal{Gr}_{\mb\mid\nb}^0]$. We have the \textit{Satake correspondence} on restricted local Hecke stacks,
\begin{equation}
     \begin{tikzcd}
    \hk_{\mb}^{\loc(m)}\arrow[r,leftarrow, "h_{\mb}^{\leftarrow}"] & \Hmbnb^{\textnormal{0,loc}(m)} \arrow[r,"h_{\nb}^{\rightarrow}"] & \hk_{\nb}^{\loc(m)}.
    \end{tikzcd}
\end{equation}
\subsection{Torsors over the local Hecke stacks}
Let $\mathbb{B}L^+G$ (resp. $BL^mG$ for $m\in\mathbb{Z}_{\geq 0}$) denote the moduli stack which classifies
for every perfect $k$-algebra $R$ the groupoid of $G$-torsors over $D_R$ (resp. $D_{m,R}$). For non-negative integers $m_1\leq m_2$, the natural quotient maps
$$
\begin{tikzcd}[sep=huge]
L^+G\arrow[r,"\res_{m_2}:=\res^{\infty}_{m_2}"] & L^{m_2}G \arrow[r, "\res^{m_2}_{m_1}"] & L^{m_1}G
\end{tikzcd}
$$
induce restriction maps between stacks
\begin{equation}
\begin{tikzcd}[sep=huge]
\mathbb{B}L^+G\arrow[r,"\res_{m_2}:=\res^{\infty}_{m_2}"] & \mathbb{B}L^{m_2}G \arrow[r, "\res^{m_2}_{m_1}"] & \mathbb{B}L^{m_1}G
\end{tikzcd}.
\end{equation}
Clearly, for any non-negative integers $m_1\leq m_2\leq m_3$, we have $\res^{m_2}_{m_1}\circ\res^{m_3}_{m_2}=\res^{m_3}_{m_1}$, where $m_3$ can be taken to be $\infty$.

Let $\mb = (\mu_1,\mu_2,\cdots, \mu_n)$ be a sequence of dominant coweights. We have natural morphisms 
$$
t_{\leftarrow},t_{\rightarrow}:\hk_{\mb}^{\loc}\rightarrow \mathbb{B}L^+G
$$
which send $(3.1)$ to the torsors $\mathcal{E}_n$ and $\mathcal{E}_0$, respectively.

For restricted local Hecke stacks, we choose a pair of $\mb$-large integers $(m,n)$. Then the natural maps 
$$
\begin{tikzcd}
\hk_{\mb}^{\loc(m)}\simeq [L^mG\backslash \textnormal{Gr}_{\leq\mb}^{(n)}/L^nG] \arrow[r,"t_{\leftarrow}\times t_{\rightarrow}"] & \mathbb{B}L^nG\times \mathbb{B}L^mG
\end{tikzcd}
$$
induce the $L^mG$-torsor
\begin{equation}
\textnormal{Gr}_{\leq\mb}\rightarrow \hk^{\loc(m)}_{\mb},
\end{equation}
and the $L^nG$-torsor
\begin{equation}
    [L^mG\backslash \textnormal{Gr}_{\leq\mb}]\rightarrow [L^{m+n}G\backslash \textnormal{Gr}_{\leq\mb}].
\end{equation}
Following the notations in \cite{xiao2017cycles}, we denote the two torsors by $\mathcal{E}_{\leftarrow}$ and $\mathcal{E}_{\rightarrow}$, respectively.
For any pairs of $\mb$-large integers $(m_1,n_1)$ and $(m_2,n_2)$ such that $m_1\leq m_2$ and $n_1\leq n_2$, denote the natural restriction map of restricted local Hecke stacks as 
\begin{equation}
    \res^{m_2}_{m_1}:\hk^{\loc(m_2)}_{\mb}\rightarrow \hk^{\loc(m_1)}_{\mb}.
\end{equation}
It is compatible with the restriction maps in $(3.4)$ in the sense that the following diagram is commutative
$$
\begin{tikzcd}[row sep=huge]
\mathbb{B}L^{n_2}G\arrow[r,leftarrow,"t_{\leftarrow}"] \arrow[d,"\res^{n_2}_{n_1}"] & \hk^{\loc(m_2)}_{\mb} \arrow[r,"t_{\rightarrow}"] \arrow[d,"\res^{m_2}_{m_1}"] & \mathbb{B}L^{m_2}G \arrow[d,"\res^{m_2}_{m_1}"] \\
\mathbb{B}L^{n_1}G \arrow[r,leftarrow,"t_{\leftarrow}"] & \hk^{\loc(m_1)}_{\mb} \arrow[r,"t_{\rightarrow}"] & \mathbb{B}L^{m_1}G.
\end{tikzcd}
$$
Let $\mb$ and $\nb$ be two sequences of dominant coweights. Choose non-negative integers $m_1,m_2,n$ such that $(m_1,m_2)$ is $\mb$-large and $(m_2,n)$ is $\nb$-large. Then we have the following isomorphism
\begin{equation}
    [L^{m_1}G\backslash Gr^{(n)}_{\leq\mb,\nb}]\cong\hk^{\loc (m_1)}_{\mb,\nb}\cong \hk^{\loc (m_2)}_{\nb}\times_{t_{\rightarrow},\mathbb{B}L^{m_2}G,\res^{m_1}_{m_2}\circ t_{\leftarrow}}\hk^{\loc (m_1)}_{\mb},
\end{equation}
which induces the following perfectly smooth morphisms
\begin{equation}
    \hk^{\loc (m_1)}_{\mb,\nb}\rightarrow  \hk^{\loc (m_2)}_{\nb}\times \hk^{\loc (m_1)}_{\mb}\xrightarrow{\textnormal{id}\times\res^{m_1}_{m_1-m_2}} \hk^{\loc (m_2)}_{\nb}\times \hk^{\loc (m_2-m_1)}_{\mb}.
\end{equation}

\subsection{Perverse sheaves on the moduli of local Hecke stacks}
Let $m_1\leq m_2$ be two $\mb$-large integers. The natural (twisted) pullback functor
$$
\res_{m_1}^{m_2}:= (\res^{m_2}_{m_1})^{\star}:=(\res^{m_2}_{m_1})^{*}[d](d/2):\textnormal{P}(\hk^{\loc(m_1)}_{\mb})\rightarrow \textnormal{P}(\hk^{\loc(m_2)}_{\mb})
$$
is an equivalence of categories.
We define the category of perverse sheaves on the local Hecke stack as
$$
\textnormal{P}(\hk_{\bar{k}}^{\loc},\Lambda):=\bigoplus_{\xi\in\pi_1(G)} \textnormal{P}(\hk_{\xi}^{\loc},\Lambda), \textnormal{P}(\hk_{\xi}^{\loc},\Lambda):=\varinjlim_{(\mu,m)\in\xi\times\mathbb{Z}_{\geq 0}}\textnormal{P}(\hk^{\loc(m)}_{\mu},\Lambda).
$$
Here the connecting morphism in the definition of $\textnormal{P}(\hk^{\loc},\Lambda)$ is the fully faithful embedding 
$$
\begin{tikzcd}[sep=huge]
\textnormal{P}(\hk^{\loc(m_1)}_{\mu_1},\Lambda) \arrow[r,"\res^{m_2}_{m_1}"] & \textnormal{P}(\hk^{\loc^{m_1}}_{\mu_1},\Lambda) \arrow[r,"i_{\mu_1,\mu_2*}"] & \textnormal{P}(\hk^{\loc(m_2)}_{\mu_2},\Lambda).
\end{tikzcd}
$$
Finally, via descent, there is a natural equivalence of categories $\textnormal{P}(\hk^{\loc(m)}_{\mu},\Lambda)\cong \textnormal{P}_{L^mG}(\textnormal{Gr}_{\leq\mu},\Lambda)$, which induces an equivalence
$\textnormal{P}(\hk_{\bar{k}}^{\loc},\Lambda)\cong \textnormal{P}_{L^+G\otimes \bar{k}}(\textnormal{Gr}_G\otimes \bar{k},\Lambda)$.

\section{Moduli of Local Shtukas}
In this section, we define moduli of local Shtukas , and its restricted version, on which we also construct various correspondences. Via these constructions, we define the category of $\Lambda$-coefficient perverse sheaves on the moduli of local Shtukas and their cohomological correspondences. In the rest of this thesis, we will make use of the theory of cohomological correspondences between perfect schemes and perfect pfp algebraic spaces. We refer to  \cite[Appendix A]{xiao2017cycles} for a detailed discussion.
\begin{definition}
 Let $\mb=(\mu_1,\mu_2,\cdots,\mu_n)$ be a sequence of dominant coweights. The \textit{moduli of local Shtukas} $\sht^{\loc}_{\mb}$ classifies for each perfect $k$-algebra $R$ sequences of modifications of $G$-torsors 
$$
\begin{tikzcd}[]
\mathcal{E}_n \arrow[r,dashed] & \mathcal{E}_{n-1} \arrow[r,dashed]& \cdots \arrow[r,dashrightarrow] & \mathcal{E}_0\cong {^{\sigma}}\mathcal{E}_n
\end{tikzcd}
$$
over $D_{R}$ of relative positions $\leq\mu_n,\cdots,\leq \mu_1$ respectively.
\end{definition}
It follows from the definition that
$$
\sht^{\loc}_{\mb}\cong \hk^{\loc}_{\mb}\times_{t_{\leftarrow}\times t_{\rightarrow},\mathbb{B}L^+G\times \mathbb{B}L^+G, \textnormal{id}\times \sigma}\mathbb{B}L^+G.
$$
There is a natural forgetful map $\psi^{\loc}:\sht^{\loc}_{\mb}\rightarrow \hk^{\loc}_{\mb}$ which forgets the isomorphisms $\mathcal{E}_0\cong {^{\sigma}}\mathcal{E}_n$. One can define the stack  
$$
\sht^{0,\loc}_{\mb\mid \nb}
$$ 
which classifies for each perfect $k$-algebra $R$ the following rectangle of modifications 
$$
\begin{tikzcd}[]
\mathcal{E}_n\arrow[r,dashed] \arrow[d,equal] &  \cdots \arrow[r,dashrightarrow] & \mathcal{E}_0\cong {^{\sigma}}\mathcal{E}_n \arrow[d,equal] \\
 \mathcal{E}'_{m} \arrow[r,dashed]& \cdots \arrow[r,dashrightarrow] & \mathcal{E}_{0}' \cong {^{\sigma}}\mathcal{E}_m',
\end{tikzcd}
$$
of $G$-torsors over $D_R$ with modifications in the upper (resp. lower) row bounded by $\mb$ (resp. $\nb$).
We get the \textit{Satake correspondence} for moduli of local Shtukas
$$
\begin{tikzcd}[column sep=huge]
\sht^{\loc}_{\mb} \arrow[r,leftarrow,"s_{\mb}^{\leftarrow}"] &  \sht^{0,\loc}_{\mb\mid \nb} \arrow[r,"s_{\mb}^{\rightarrow}"] & 
\sht^{\loc}_{\nb}
\end{tikzcd}.
$$
We introduce the partial Frobenius morphism between the moduli of local Shtukas which will play an important role in later constructions. 
\begin{definition}
Let $\mb=(\mu_1,\mu_2,\cdots,\mu_n)$ be a sequence of dominant coweights of $G$. We define the \textit{partial Frobenius morphism} to be
\begin{equation}
\begin{tikzcd}
F_{\mu_{\bullet}}:\sht^{\loc}_{\mu_1,\cdots,\mu_n}\arrow[r] & \sht^{\loc}_{\sigma(\mu_n),\mu_1,\cdots,\mu_{n-1}} 
\end{tikzcd}
\end{equation}
$$
\begin{tikzcd}[sep=small]
(\mathcal{E}_n \arrow[r,dashed] & \cdots \arrow[r,dashrightarrow] & \mathcal{E}_0\cong {^{\sigma}}\mathcal{E}_n )\arrow[r,mapsto] &(\mathcal{E}_{n-1} \arrow[r,dashed] & \cdots \arrow[r,dashrightarrow] & {^{\sigma}}\mathcal{E}_n\arrow[r,dashrightarrow] & {^{\sigma}}\mathcal{E}_{n-1}).
\end{tikzcd}
$$
\end{definition}
\begin{definition}
Let $\mb$ and $\nb$ be two sequences of dominant coweights. For each perfect $k$-algebra $R$, the prestack $\Smbnb^{\loc}$ classifies the following commutative diagram of modifications of $G$-torsors over $D_{R}$
$$
\begin{tikzcd}[row sep=huge]
\mathcal{E}_n\arrow[r,dashed] \arrow[d,dashed,"\beta"] &  \cdots \arrow[r,dashrightarrow] & \mathcal{E}_0\cong {^{\sigma}}\mathcal{E}_n \arrow[d,dashed,"\beta^{\sigma}"] \\
 \mathcal{E}'_{m} \arrow[r,dashed]& \cdots \arrow[r,dashrightarrow] & \mathcal{E}'_0 \cong {^{\sigma}}\mathcal{E}_m',
\end{tikzcd}
$$
where the top (resp. bottom) row defines an $R$-point of $\sht_{\mb}^{\loc}$ (resp. $\sht_{\nb}^{\loc}$). Let $\overleftarrow{h}^{\loc}_{\mb}$ (reps. $\overrightarrow{h}^{\loc}_{\nb}$) denote the morphism which maps the above commutative rectangle to its upper (resp. lower) row. We define the \textit{Hecke correspondence} of local Shtukas to be to following diagram
\begin{equation}
\begin{tikzcd}[column sep=huge]
\sht_{\mb}^{\loc} \arrow[r,leftarrow,"\overleftarrow{h}^{\loc}_{\mb}"] & \Smbnb^{\loc} \arrow[r, "\overrightarrow{h}^{\loc}_{\nb}"] & \sht_{\nb}^{\loc}.
 \end{tikzcd}
    \end{equation}
\end{definition}
If in addition, the relative position of $\beta$ is bounded by $\lambda$, we get a closed sub-prestack $\sht^{\lambda,\loc}_{\mb\mid\nb}$. In particular, if $\lambda=0$, the Hecke correspondence $(4.2)$ reduces to the Satake correspondence.

The Hecke correspondence can be considered as the composition of two Satake correspondences and the cohomological correspondence given by the partial Frobenius morphism. More precisely, we recall \cite[Lemma 5.2.14]{xiao2017cycles}.
\begin{lemma}
Let $\mb$ and $\nb$ be two sequences of dominant coweights. Choose $\lambda$ to be a dominant coweight such that $\lambda\geq  |\mb| +\sigma(\lambda) $ or $\lambda\geq |\mb|+\nu$, then we have the following commutative diagram of prestacks
$$
\begin{tikzcd}[]
 &  & \sht^{\theta,\loc}_{\mb|\nb} \arrow[dl] \arrow[dr] &  & \\
& \sht_{\mb|(\sigma(\theta^*),\lambda)}^{0,\loc} \arrow[dl,"s^{\leftarrow}_{\mb}"'] \arrow[d,"s_{\sigma(\theta^*),\lambda}^{\rightarrow}"] & & \sht_{(\lambda,\theta^*)|\nb}^{0,\loc} \arrow[dr,"s_{\nb}^{\rightarrow}"] \arrow[d,"s_{\lambda,\theta^*}^{\leftarrow}"] &\\
\sht_{\mb}^{\loc} & \sht_{\sigma(\theta^*),\lambda} \arrow[rr,"F_{\lambda,\theta^*}^{-1}"] & & \sht_{\lambda,\theta^*}^{\loc}& \sht_{\nb}^{\loc}. \\
\end{tikzcd}
$$
In addition, the pentagon in the middle is a Cartesian square when composing  $F_{\lambda,\theta^*}$ with $s_{\sigma(\theta^*),\lambda}^{\rightarrow}$.
\end{lemma}

\subsection{Moduli of restricted local Shtukas}
Let $\mb=0$ or more generally, a central cocharacter, $\sht_{\mb}^{\loc}\simeq \mathbb{B}G(
\mathcal{O})$ which is  not perfectly of finite presentation as a prestack. Thus to apply the $\ell$-adic formalism, it is desirable to study the following finite type approximation of $\sht^{\loc}_{\mb}$.
\begin{definition}
Let $\mb$ be a sequence of dominant coweights and $(m,n)$ a pair of $\mb$-large integers. We define the moduli stack $\sht^{\loc(m,n)}_{\mb}$ of $(m, n)$-\textit{restricted
local iterated shtukas} as the stack that classifies for every perfect $k$-algebra $R$, 
\begin{itemize}
    \item [(1)] an $R$-point of $\hk^{\loc (m)}_{\mb}$,
    \item [(2)] an isomorphism 
    $$
    \Psi:{^{\sigma}}(\mathcal{E}_{\leftarrow}\mid_{D_{n,R}})\simeq (\mathcal{E}_{\rightarrow}\mid_{D_{m,R}})\mid_{D_{n,R}}
    $$
    of $L^n G$-torsos over $\textnormal{Spec} R$, where $\mathcal{E}_{\leftarrow}$ and $\mathcal{E}_{\rightarrow}$ are defined in $(3.5)$ and $(3.6)$, respectively.
\end{itemize}
\end{definition}
The above definition gives a canonical isomorphism 
$$\sht_{\mb}^{\loc (m,n)}\cong \hk_{\mb}^{\loc (m)}\times_{t_{\leftarrow}\times \res^{m}_n\circ t_{\rightarrow},\mathbb{B}L^nG\times \mathbb{B}L^nG,\textnormal{id}\times \sigma}\mathbb{B}L^nG.
$$
The natural forgetful morphism $\psi^{\loc (m,n)}:\sht^{\loc (m,n)}_{\mb}\rightarrow \hk^{\loc (m)}_{\mb}$ is a perfectly smooth
morphism of relative dimension $n\dim G$. For two sequences of dominant coweights $\mb$, $\nb$, we define $\Smbnb^{\loc (m,n)}$ to be the stack which classifies for each perfect $k$-algebra $R$, an $R$-point of $\Hmbnb^{\loc (m,n)}$ together with an isomorphism ${^{\sigma}}(\mathcal{E}_{\leftarrow}\mid_{D_{n,R}})\simeq (\mathcal{E}\mid_{D_{m,R}})\mid_{D_{n,R}}$. Let $(m_1
, n_1)$ and $(m_2, n_2)$ be two pairs of $\mb$-large integers such that $m_1\leq m_2$ and $n_1\leq n_2$. We define the restriction morphism 
\begin{equation}
    \res^{m_2,n_2}_{m_1.n_1}:\sht^{\loc (m_2,n_2)}_{\mb}\rightarrow \sht^{\loc (m_1,n_1)}_{\mb}
\end{equation}
as the composition of the following morphisms
\begin{align*}
\res^{m_2,n_2}_{m_1.n_1}: & \sht_{\mb}^{\loc (m_2,n_2)}\cong  \hk_{\mb}^{\loc (m_2)}\times_{t_{\leftarrow}\times \res^{m_2}_{n_2}\circ t_{\rightarrow},\mathbb{B}L^{n_2}G\times \mathbb{B}L^{n_2}G,\textnormal{id}\times \sigma}\mathbb{B}L^{n_2}G \\
& \xrightarrow[]{\res^{m_2}_{m_1}\times\res^{n_2}_{n_1} } \hk_{\mb}^{\loc (m_1)}\times_{t_{\leftarrow}\times \res^{m_1}_{n_1}\circ t_{\rightarrow},\mathbb{B}L^{n_1}G\times \mathbb{B}L^{n_1}G,\textnormal{id}\times \sigma}\mathbb{B}L^{n_1}G\\
 & \cong \sht_{\mb}^{\loc (m_1,n_1)}.
\end{align*}
For $(m_2,n_2)=(\infty,\infty)$, we write 
$\res^{m_2,n_2}_{m_1.n_1}$ as $\res_{m_1,n_1}$ for simplicity. 
For three pairs of $\mb$-large integers $(m_i,n_i)$ such that $m_1\leq m_2\leq m_3$ and $n_1\leq n_2 \leq n_3$, we have
\begin{equation}
    \res^{m_2,n_2}_{m_1,n_1}\circ \res^{m_3,n_3}_{m_2,n_2}=\res^{m_3,n_3}_{m_1,n_1}.
\end{equation}
The Satake correspondences for restricted local Hecke stacks and the Satake correspondences for restricted local Shtukas are related by the restriction and forgetful morphisms. We summarize this relation in following diagram
\begin{equation}
\begin{tikzcd}[sep=small]
\sht_{\mb}^{\loc(m_2,n_2)} \arrow[dd,"\psi^{\loc(m_2,n_2)}", near end] \arrow[dr,"\res_{m_1,n_1}^{m_2,n_2}"]  &  & \sht_{\mb|\nb}^{0,\loc(m_2,n_2)} \arrow[ll] \arrow[rr] \arrow[dr,"\res_{m_1,n_1}^{m_2,n_2}"] \arrow[dd,dashrightarrow,"\psi^{\loc(m_2,n_2)}", near end] &  & \sht_{\nb}^{\loc(m_2,n_2)} \arrow[dr,"\res_{m_1,n_1}^{m_2,n_2}"] \arrow[dd,dashrightarrow,"\phi^{\loc(m_2,n_2)}", near end] & \\
 & \sht_{\mb}^{\loc(m_1,n_1)} \arrow[dd,"\psi^{\loc(m_1,n_1)}", near end] &  & \sht_{\mb|\nb}^{0,\loc(m_1,n_1)} \arrow[ll] \arrow[rr] \arrow[dd,"\psi^{\loc(m_1,n_1)}", near end] &  & \sht_{\nb}^{\loc(m_1,n_1)} \arrow[dd,"\psi^{\loc(m_1,n_1)}", near end]   \\
\hk_{\mb}^{\loc(m_2)} \arrow[dr] &  & \hk_{\mb|\nb}^{0,\loc(m_2)} \arrow[ll,dashrightarrow] \arrow[dr,dashrightarrow] \arrow[rr,dashrightarrow] &  & \hk_{\nb}^{\loc(m_2)} \arrow[dr,dashrightarrow] & \\
 & \hk_{\mb}^{\loc(m_1)} &  & \hk_{\mb|\nb}^{0,\loc(m_1)} \arrow[ll] \arrow[rr] &  & \hk_{\nb}^{\loc(m_1)}\\
\end{tikzcd}
\end{equation}
where 
\begin{itemize}
    \item [(1)] all rectangles are commutative,
    \item [(2)] all rectangles are Cartesian except for the two on the left and right side of the cuboid.
\end{itemize}
Let $\mb=(\mu_1,\cdots,\mu_n)$ be a sequence of dominant cocharacters. We call a quadruple of non-negative integers $(m_1, n_1, m_2, n_2)$ $\mb$-\textit{acceptable} if
\begin{itemize}
    \item [(1)]  $m_1 -m_2 = n_1-n_2$ are $\mu_n$-large (or equivalently $\sigma (\mu_n)$-large),
    \item[(2)]  $m_2 - n_1$ is $\mb$-large.
\end{itemize}
We can define the partial Frobenius morphism 
\begin{equation}
    F_{\mb}^{-1}:\sht_{\sigma(\mu_n),\mu_1,\cdots,\mu_{n-1}}^{\loc (m_1,n_1)}\rightarrow\sht_{\mu_1,\cdots,\mu_2}^{\loc (m_2,n_2)}
\end{equation}
for restricted local Shtukas. The construction of $F_{\mb}^{-1}$ is technical and we refer to \cite[Construction $5.3.12$]{xiao2017cycles} for detailed discussion. 

\subsection{Perverse sheaves on the moduli of local Shtukas}
Let $\mb$ be a sequence of dominant coweights
and $(m_1,n_1)$, $(m_2,n_2)$ be two pairs of $\mb$-large integers such that $m_1\leq m_2$,  $n_1\leq n_2$, and $m_2\neq \infty$. Define the functor 
\begin{equation}
\res^{m_2,n_2}_{m_1,n_1}:=(\res^{m_2,n_2}_{m_1.n_1})^{\star}: \textnormal{P}(\sht^{\loc (m_2,n_2)}_{\mb},\Lambda)\rightarrow \textnormal{P}(\sht^{\loc (m_1,n_1)}_{\mb,\Lambda}).
\end{equation}
Then $(4.4)$ yields  
\begin{equation}
\res^{m_2,n_2}_{m_1,n_1} \circ\res^{m_3,n_3}_{m_2,n_3}=\res^{m_3,n_3}_{m_1,n_1}.
\end{equation}
Like $\res^{m}_n$, the functor $\res^{m_i,n_i}_{m_j,n_j}$ is also an equivalence of categories if $m_j>1$.

We define the category of perverse sheaves on the moduli of local Shtukas as 
\begin{equation}
    \textnormal{P}(\sht^{\loc}_{\bar{k}},\Lambda):=\bigoplus_{\xi\in\pi_1(G)}\textnormal{P}(\sht^{\loc}_{\xi},\Lambda),\textnormal{ }\textnormal{P}(\sht^{\loc}_{\xi},\Lambda):=\varinjlim_{(m,n,\mu)}\textnormal{P}(\sht^{\loc (m,n)}_{\mu},\Lambda)
\end{equation}
where the limit is taken over the triples $\{(m,n,\mu) \in\mathbb{Z}^2\times \xi\mid (m,n)\textnormal{ is }\mu\textnormal{ large}\}$  with the product
partial order. As in \cite{xiao2017cycles}, we call objects in $\textnormal{ }\textnormal{P}(\sht^{\loc}_{\xi},\Lambda)$ \textit{connected objects}. The connecting morphism is given by the composite of fully faithful functor
$$
\begin{tikzcd}[sep=huge]
\textnormal{P}(\sht^{\loc(m_1,n_1)}_{\mu_1},\Lambda) \arrow[r,"\res^{m_2,n_2}_{m_1,n_1}"] & \textnormal{P}(\sht^{\loc(m_2,n_2)}_{\mu_1},\Lambda) \arrow[r,"i_{\mu_1,\mu_1'}"] & \textnormal{P}(\sht^{\loc(m_2,n_2)}_{\mu_1'},\Lambda).
\end{tikzcd}
$$
For each dominant coweight $\mu$ and a pair of $\mu$-large integers $(m,n)$, we define the natural pullback functor 
\begin{equation}
    \Psi^{\loc (m,n)}:=\res^{m,n}_{m,0}: \textnormal{P}(\hk^{\loc (m)}_{\mu},\Lambda)\rightarrow \textnormal{P}(\sht^{\loc (m,n)}_{\mu},\Lambda).
\end{equation}
We observe that $\Psi^{\loc (m,n)}$ commutes with the connecting morphism in $(4.9)$ by  $(4.8)$ and the proper smooth base change. Then we can take the limit and direct sum of $\Psi^{\loc (m,n)}$ and derive the following well-defined functor
\begin{equation}
    \Psi^{\loc}:\textnormal{P}(\hk_{\bar{k}}^{\loc},\Lambda)\rightarrow \textnormal{P}(\sht^{\loc}_{\bar{k}},\Lambda).
\end{equation}
Let $\mathcal{F}_i\in \textnormal{ }\textnormal{P}(\sht^{\loc}_{\xi_i},\Lambda)$ be connected objects. It is realized as $\mathcal{F}_{i,\mu_i}^{(m_i,n_i)}\in\textnormal{P}(\sht_{\mu_i}^{\loc (m_i,n_i)},\Lambda)$ for some $\mu_i$ and some pair of $\mu_i$-large integers $(m_i,n_i)$. We define the set of cohomological correspondences between $\mathcal{F}_1$ and $\mathcal{F}_2$ as
\begin{align*}
    & \cc_{\sht^{\loc}}(\mathcal{F}_1,\mathcal{F}_2)\\ \notag := & \bigoplus_{\xi\in\pi_1(G)}\varinjlim \cc_{\sht_{\mu_1|\mu_2}^{\lambda,\loc (m_1,n_1)}}\Big ((\sht_{\mu_1}^{\loc (m_1,n_1)},\mathcal{F}_{1,\mu_1}^{(m_1,n_1)}),(\sht_{\mu_2}^{\loc (m_2,n_2)},\mathcal{F}_{2,\mu_2}^{(m_2,n_2)})\Big ),
\end{align*}
where the limit is taken over all partially ordered sextuples $(\mu_1,\mu_2,\lambda,m_1,n_1,m_2,n_2)$ such that \begin{itemize}
    \item $(m_1,n_1,m_2,n_2)$ is $(\mu_1+\lambda,\lambda)$ and $(\mu_2+\lambda,\lambda)$-acceptable,
    \item $\mu_i\in\xi_i$, for some $\xi_i\in\pi_1(G)$,
    \item $\lambda\in\xi$.
\end{itemize}
Let $(\mu_1,\mu_2,\lambda,m_1,n_1,m_2,n_2)\leq (\mu_1',\mu_2',\lambda',m_1',n_1',m_2',n_2')$ be another such sextuple. The connecting morphism between the cohomological correspondences 
\begin{equation}
\cc_{\sht_{\mu_1|\mu_2}^{\lambda,\loc (m_1,n_1)}}\Big ((\sht_{\mu_1}^{\loc (m_1,n_1)},\mathcal{F}_{1,\mu_1}),(\sht_{\mu_2}^{\loc (m_2,n_2)},\mathcal{F}_{2,\mu_2})\Big )
\end{equation}
and 
\begin{equation}
\cc_{\sht_{\mu_1'|\mu_2'}^{\lambda',\loc (m_1',n_1')}}\Big ((\sht_{\mu_1'}^{\loc (m_1',n_1')},\mathcal{F}_{1,\mu_1'}),(\sht_{\mu_2'}^{\loc (m_2',n_2')},\mathcal{F}_{2,\mu_2'})\Big )
\end{equation}
is given by first pulling back $(4.13)$ to the Hecke correspondence
$$
\begin{tikzcd}
\sht_{\mu_1}^{\loc (m_1',n_1')} \arrow[r,leftarrow] & \sht^{\lambda,\loc(m_1',n_1')}_{\mu_1|\mu_2} \arrow[r, ] & \sht_{\mu_2}^{\loc(m_2',n_2') },
 \end{tikzcd}
$$
along the restriction morphism, then pushing it forward to the Hecke correspondence
$$
\begin{tikzcd}
\sht_{\mu_1'}^{\loc (m_1',n_1')} \arrow[r,leftarrow] & \sht^{\lambda',\loc(m_1',n_1')}_{\mu_1'|\mu_2'} \arrow[r, ] & \sht_{\mu_2'}^{\loc(m_2',n_2') }.
 \end{tikzcd}
$$
The connecting morphism is well-defined and can be composed. We refer to \cite[\S 5.4.1]{xiao2017cycles} for more discussions.

\section{Key theorem}
In this section, we state and prove the key theorem for our construction of the Jacquet-Langlands transfer. We will make use of the theory of the cohomological correspondences  throughout this chapter. Instead of explaining all the details, we refer to \cite[Appendix A.2]{xiao2017cycles} for a nice discussion.
\subsection{Set up}
Fix a half Tate twist $\Lambda(1/2)$. Recall notations $\langle  d \rangle$ and $f^{\star}$ introduced in \S$1.3$. Throughout this section, we consider the Langlands dual group scheme $\hat{G}_\Lambda$ over $\Lambda$ of $G$ and its $\Lambda$-representations. The subscripts $\Lambda$ will be omitted for simplicity. We generalize a few notions introduced in previous sections for the sake of stating the key theorem.

\subsection{More on local Hecke stacks}
Let $\vb:=V_1 \boxtimes V_2\boxtimes\cdots \boxtimes V_s\in \textnormal{Rep}(\hat{G}^s)$ and assume that for each $i$, $V_i$ has the Jordan-H$\ddot{\textnormal{o}}$lder factors $\{V_{\mu_{ij}}\}_j$.

The integral geometric Satake equivalence (cf. Theorem $2.5$) $\textnormal{Sat}_{G^s}$ sends $\vb$ to an $(L^+G\otimes \bar{k})^s$-equivariant perverse sheaf $\textnormal{Sat}_{G^s}(\vb)$ on $(\textnormal{Gr}_G\otimes\bar{k})^s$. We write $\textnormal{Gr}_{\vb}$ for the support of the external tensor product $\textnormal{Sat}(V_1)\tilde{\boxtimes}\textnormal{Sat}(V_2)\tilde{\boxtimes}\cdots\tilde{\boxtimes}\textnormal{Sat}(V_s)$. Let $m$ be a non-negative integer. We call it $V_i$-\textit{large} if $m$ is $\mu_{ij}$-large for each $j$, and we call it $\vb$-\textit{large} if $m=m_1+m_2+\cdots+m_s$ such that $m_i$ is $V_i$-large for each $i$. For a $\vb$-large integer $m$, $\textnormal{Sat}_{G^n}(\vb)$ descends to a perverse sheaf supported on $\hk^{\loc (m)}_{\vb}:=[L^mG\backslash \textnormal{Gr}_{\vb}]$. We write  $S(V)^{\loc (m)}$ for the twist of this perverse sheaf by $\langle m\dim G \rangle$. Note that $S(\vb)^{\loc (m)}$ is isomorphic to the "$\star$"-pullback of $S(V_1)^{\loc (m_1)}\boxtimes S(V_2)^{\loc (m_2)}\boxtimes\cdots\boxtimes S(V_s)^{\loc (m_s)}$ along the perfectly smooth morphism $\hk_{\vb}^{\loc (m)}\rightarrow \prod_i \hk^{\loc (m_i)}_{V_i}$ constructed in $(3.9)$.

In the case $s=1$, we have 
$$
\textnormal{Gr}_{V_1}=\cup_{j}\textnormal{Gr}_{\mu_{1j}}, \text{ }\hk_{V_1}^{\loc (m)}=\cup_{j}\hk^{\loc m}_{\mu_{1j}}.
$$  
In general, $\hk^{\loc (m)}_{\vb}$ is of the form $\cup_{\mb}\hk^{\loc (m)}_{\mb}$. Via descent, Corollary $2.7$ gives the following natural isomorphism: 
\begin{equation}
    \hg(\vb,\wb)\cong \cc_{\hk_{\vb\mid\wb}^{0,\loc (m)}}\big ((\hk_{\vb}^{\loc (m)},\textnormal{Sat}_{G}^{\loc (m)}(\vb)),(\hk_{\wb}^{\loc (m)},\textnormal{Sat}_{G}^{\loc (m)}(\wb))\big).
\end{equation}
Here and below, we regard $\vb$ and $\wb$ as representations of $\hat{G}$ via the diagonal embedding $\hat{G}\hookrightarrow \hat{G}^s$.

Let $\vb$ and $\wb$ be two representations of $\hat{G}^s$. We can similarly define $\textnormal{Gr}_{\vb\mid \wb}^0:= \textnormal{Gr}_{\vb}\times_{\textnormal{Gr}_G}\textnormal{Gr}_{\wb}$ and $\hk_{\vb\mid \wb}^{0,\loc (m)}=[L^mG \backslash \textnormal{Gr}_{\vb\mid\wb}^0]$, and Satake correspondences supported on them.
\subsection{More on moduli of local Shtukas}
Let $\vb\in\textnormal{Rep}(\hat{G}^s)$. For a pair of non-negative integers $(m,n)$,  we can generalize the notion of $\mb$-\textit{large} and define the notion of $\vb$-\textit{large}.  Let $(m,n)$ be a pair of $\vb$ and $\wb$-large integers, we can define the moduli of restricted local Shtukas $\sht_{\vb}^{\loc (m,n)}$ and $\sht_{\vb\mid\wb}^{\loc (m,n)}$. Similar to $\hk^{\loc}_{\vb}$, the stacks $\sht_{\vb}^{\loc (m,n)}$ and $\sht_{\vb\mid\wb}^{\loc (m,n)}$ can be regarded as unions of $\sht^{\loc (m,n)}_{\mb}$ and unions of $\sht_{\mb\mid\nb}^{\loc (m,n)}$. We have the natural forgetful map
\begin{equation}
    \psi^{\loc (m,n)}:\sht_{\vb}^{\loc (m,n)}\rightarrow \hk_{\vb}^{\loc (m)}.\end{equation}
Choose a pair of $\vb$-large integers $(m,n)$ such that $n>0$. Write 
$$
S(\widetilde{\vb})^{\loc (m,n)}:=\Psi^{\loc (m,n)}(\textnormal{Sat}(\vb)^{\loc (m)})\in \textnormal{P}(\sht^{\loc (m,n)},\Lambda)
$$
for the pullback of $\textnormal{Sat}(V)^{\loc (m)}$ along the morphism $\psi^{\loc (m,n)}$ (up to a shift and twist). For $s=1$, $S(\widetilde{V})^{\loc (m,n)}$ represents the perverse sheaf $S(\widetilde{V}):=\Psi(\textnormal{Sat}_{G}(V))\in \textnormal{P}(\sht^{\loc}_{\bar{k}},\Lambda)$.

Consider the front face of the diagram $(4.5)$. The middle and right vertical maps are perfectly smooth. Pulling back the cohomological correspondence on the right hand side of $(5.1)$ to the upper edge and pre-composing it with $(5.1)$, we get the map 
\begin{equation}
\mathcal{C}^{\loc (m,n)}:  \hg(\vb,\wb)\rightarrow \cc_{\sht_{\vb\mid\wb}^{0,\loc (m,n)}}(S(\widetilde{\vb})^{\loc (m,n)},S(\widetilde{\wb})^{\loc (m,n)}).  
\end{equation}
The map $\mathcal{C}^{\loc (m,n)}$ is compatible with the compositions at the source and target, and we refer to \cite[Lemma 6.1.8]{xiao2017cycles} for the proof.

Let $\vb\in\textnormal{Rep}(\hat{G}^s)$ and $W\in\textnormal{Rep}(\hat{G})$. We call a
quadruple of non-negative integers $(m_1, n_2, m_1, n_1)$ $\vb\boxtimes W$-\textit{acceptable} if 
\begin{itemize}
    \item $m_1-m_2=n_1-n_2$ is $W$-large,
    \item $(m_2,n_1)$ is $\vb$-large.
\end{itemize}
For a quadruple of $\vb\boxtimes W$-acceptable integers $(m_1, n_2, m_1, n_1)$, we can construct the partial Frobenius morphism 
\begin{equation}
    F_{\vb\boxtimes W}^{-1}:\sht^{\loc (m_1,n_1)}_{\sigma W\boxtimes \vb}\rightarrow \sht^{\loc (m_2,n_2)}_{\vb\boxtimes W}
\end{equation}
similar to $(4.1)$. Here, $\sigma W$ is the Frobenius twist of $W$ as in Remark $2.6$.

Let $V_1,V_2\in \textnormal{Rep}(\hat{G})$. For any projective object $W\in \textnormal{Rep}(\hat{G})$, choose a quadruple of $((V_1\otimes V_2\otimes W)\boxtimes W^*)$-acceptable integers $(m_1,n_1,m_2,n_2)$. We define the following stack
\begin{equation}
    \sht^{W,\loc (m_1,n_1)}_{V_1\mid V_2}:=\sht^{\loc (m_1,n_1)}_{V_1\mid \sigma W^*\boxtimes (\sigma W\otimes V_1)}\times_{\sht^{\loc (m_2,n_2)}_{(\sigma W\otimes V_1)\boxtimes W^*}} \sht^{\loc (m_1.n_1)}_{(\sigma W\otimes V_1)\boxtimes W^*\mid V_2}.
\end{equation}

\subsection{The category of coherent sheaves on the stack of unramified Langlands parameters}
Recall from Remark $2.6$ that the Langlands dual group $\hat{G}$ is naturally equipped with an action of the arithmetic Frobenius $\sigma$. Consider the $\sigma$-twisted conjugation action of $\hat{G}$ on $\hat{G} $. We denote by
$\textnormal{Coh}^{\hat{G} }(\hat{G} \sigma)$
the abelian category of $\hat{G} $-equivariant coherent sheaves on the (non-neutral) component $\hat{G} \sigma\subset\hat{G} \rtimes \sigma$. Equivalently, $\textnormal{Coh}^{\hat{G} }(\hat{G} \sigma)$ can be regarded as the abelian category of coherent sheaves on the quotient stack $[\hat{G} \sigma/\hat{G} ]$ where $\hat{G} $ acts on $\hat{G} \sigma$ by the usual conjugation action.

Let $V\in \textnormal{Rep}(\hat{G})$
 be an algebraic representation of $\hat{G}$. There is an associated vector bundle on $\hat{G}\sigma$ with global section $\mathcal{O}_{\hat{G}}\otimes V$. Consider the following action of $\hat{G}$ on $\mathcal{O}_{\hat{G}}\otimes V$. For any $g\in\hat{G}$ and $(f,v)\in\mathcal{O}_{\hat{G}}\otimes V$, $g\cdot (f,v):=(gf\sigma^{-1}(g),gv)$. The associated vector bundle thus gives an object $\widetilde{V}\in\textnormal{Coh}^{\hat{G}}(\hat{G}\sigma)$. 
 \subsection{Key theorem}
 The following theorem is an analogue of \cite[Theorem 6.0.1]{xiao2017cycles}.
\begin{theorem}
Let $V_1,V_2\in\textnormal{Rep}_{\Lambda}(\hat{G}_{\Lambda})$ be two projective $\Lambda$-modules.
Then  there exists the following map
\begin{equation}
    \sv:\textnormal{Hom}_{\textnormal{Coh}^{\hat{G}}(\hat{G}\sigma)}(\widetilde{V_1},\widetilde{V_2})\longrightarrow \cc_{\sht^{\loc}}( S(\widetilde{V_1}), S(\widetilde{V_2})),
\end{equation}
which is compatible with the natural composition maps in the source and target.
\end{theorem}
We prove this theorem in the rest of this section.

\subsubsection{} We give an explicit construction of $\sv$. Consider the following canonical isomorphisms
\begin{align}
   & \textnormal{Hom}_{\textnormal{Coh}^{\hat{G}}(\hat{G}\sigma)}(\widetilde{V_1},\widetilde{V_2}) \\ \notag
   \cong & \textnormal{Hom}_{\mathcal{O}_{\hat{G}\sigma}}(\mathcal{O}_{\hat{G}\sigma}\otimes V_1, \mathcal{O}_{\hat{G}\sigma}\otimes V_2)^{\hat{G}}  \\ \notag
   \cong & \textnormal{Hom}(V_1,\mathcal{O}_{\hat{G}\sigma}\otimes V_2)^{\hat{G}} \\ \notag
   \cong & (V_1^*\otimes \mathcal{O}_{\hat{G}\sigma}\otimes V_2)^{\hat{G}}.
\end{align}
Let $W\in\textnormal{Rep}_{\Lambda}(\hat{G}_{\Lambda})$ be a projective $\Lambda$-module with $\Lambda$-basis $\{e_i\}_i$ and dual basis $\{e_i^*\}_i$. We construct the map 
$$
\Theta_W:\textnormal{Hom}_{\hat{G}_{\Lambda}}(V_1,\sigma W^*\otimes V_2\otimes W)\cong \hg(V_1,\textnormal{Hom}(\sigma W\otimes W^*,V_2))\rightarrow \textnormal{Hom}_{\textnormal{Coh}^{\hat{G}}(\hat{G}\sigma)}(\widetilde{V_1},\widetilde{V_2}),
$$
by sending $\mathbf{a}\in \textnormal{Hom}_{\hat{G}_{\Lambda}}(V_1,\sigma W^*\otimes V_2\otimes W)$ to the $V_1^*\otimes V_2$-valued function $\Theta_W(\mathbf{a})$ on $\hat{G}\sigma$ defined by 
$$
(\Theta_W(\mathbf{a})(g))(v_1):=\sum_i (\mathbf{a}(v_1))(ge_i\otimes e_i^*).
$$

It suffices to construct the map
$$
\mathcal{C}_{W}: \hg (V_1, \sigma W^*\otimes W\otimes V_2) \rightarrow \cc_{\sht^{\loc}}( S(\widetilde{V_1}), S(\widetilde{V_2})).
$$
for every $W\in \textnormal{Rep}_{\Lambda}(\hat{G}_{\Lambda})$.
Let $\mathbf{a}\in \hg (V_1, \sigma W^*\otimes W\otimes V_2)$.
We have the following coevaluation and evaluation maps: 
$$
\delta_{\sigma W}:\textbf{1}\rightarrow \sigma W^*\otimes \sigma W, e_W:W\otimes W^*\rightarrow \textbf{1}.
$$
Choose a quadruple $(m_1,n_1,m_2,n_2)$ of $(V_1\otimes V_2\otimes W)\boxtimes W^*$-large integers. 
Then the map $\mathcal{C}^{\loc (m_1,n_1)}$ defined in $(5.3)$ sends $\mathbf{a}$ to the cohomological correspondence 
\begin{equation}
  \mathcal{C}^{\loc (m_1,n_1)}(\mathbf{a}): S(\widetilde{V}_1)^{\loc (m_1,n_1)}\longrightarrow S(\sigma\widetilde{W^*}\boxtimes (\widetilde{V_2}\otimes\widetilde{W}))^{\loc (m_1,n_1)}.
\end{equation}
The partial Frobenius morphism $(5.4)$ gives rise to the cohomological correspondence (cf.\cite[A.2.3]{xiao2017cycles})
\begin{equation}
  \mathbb{D}\Gamma^*_{F_{(W\otimes V_2)\boxtimes W^*}^{-1}}: S(\sigma\widetilde{W^*}\boxtimes (\widetilde{V_2}\otimes\widetilde{W}))^{\loc (m_1,n_1)}\longrightarrow S((\widetilde{V_2}\otimes\widetilde{W})\boxtimes \widetilde{W^*})^{\loc (m_2,n_2)}.
\end{equation}
Finally, $\mathcal{C}^{\loc (m_2,n_2)}$ sends $\textnormal{id}\otimes e_W$ to the cohomological correspondence 
\begin{equation}
  \mathcal{C}^{\loc (m_2,n_2)}(\textnormal{id}\otimes e_{W}):  S((\widetilde{V_2}\otimes \widetilde{W})\boxtimes \widetilde{W^*})^{\loc (m_2,n_2)}\longrightarrow S(\widetilde{V_2})^{\loc (m_2,n_2)}.
\end{equation}
The composition of cohomological correspondences $(5.8)$, $(5.9)$, and $(5.10)$ yields a cohomological correspondence 
$$
\mathcal{C}_{W}(\mathbf{a})\in \cc_{\sht^{W,\loc (m_1,n_1)}_{V_1\mid V_2}}(S(\widetilde{V_1})^{\loc (m_1,n_1)},S({\widetilde{V_2})}^{\loc (m_2,n_2)}).
$$

The construction of the map $\sv$ can be summarized in the following diagram
$$
\begin{tikzcd}[row sep=huge]
\textnormal{Hom}_{\textnormal{Coh}^{\hat{G}}(\hat{G}\sigma)}(\widetilde{V_1},\widetilde{V_2}) \arrow[rr,dashed, "\sv"] &  &\cc_{\sht^{\loc}}( S(\widetilde{V_1}), S(\widetilde{V_2}))
\\
& \hg (V_1, \sigma W^*\otimes  V_2 \otimes W) \arrow[ul,"\Theta_W"] \arrow[ur,"\mathcal{C}_{W}"] &.
\end{tikzcd} 
$$
\subsubsection{}We prove that the cohomological correspondence constructed in the previous section is well-defined and can be composed.

Let $\mathbf{a'}$ denote the image of $\mathbf{a}$ under the canonical isomorphism $\hg (V_1,\sigma W^*\otimes V_2 \otimes W )\cong \hg(\sigma W\otimes V_1 \otimes W^*,V_2 )$. 
\begin{lemma}
Let $X,Y,W_1,W_2,W_1',W_2'$ be representations of $\hat{G}$, and $f_1\otimes f_2:W_1\otimes W_2\rightarrow W_1'\otimes W_2'$ be a $\hat{G}\times \hat{G}$-module homomorphism. Let $\mathbf{b}\in \hg(X,\sigma W_1\otimes Y\otimes W_2)$ and $\mathbf{b'}\in \hg(Y\otimes W_2'\otimes W_1',Y)$. We omit choosing appropriate integers $(m_i,n_i)$ for simplicity. Then we have 
\begin{equation}
    \mathcal{C}(\mathbf{b'}\circ (\textnormal{id}\otimes f_2\otimes f_1))\circ \mathbb{D}\Gamma_{F^{-1}}^*\circ \mathcal{C}(\mathbf{b})=\mathcal{C}(\mathbf{b'})\circ \mathbb{D}\Gamma_{F^{-1}}^*\circ \mathcal{C}((\sigma f_1\circ \textnormal{id}\otimes f_2)\circ \mathbf{b}).
\end{equation}
In particular, the cohomological correspondence $\sv(\mathbf{a})$ equals to the composition of the following cohomological correspondences:
\begin{align*}
    & \mathcal{C}(\delta_{\sigma W}\otimes \textnormal{id}_{V_1}):S(\widetilde{V_1})\longrightarrow S(\sigma\widetilde{W^*}\boxtimes (\sigma\widetilde{W}\otimes \widetilde{V_1})),\\
    & \mathbb{D}\Gamma^*_{F_{(W\otimes V_1)\boxtimes W^*}^{-1}}: S(\sigma\widetilde{W^*}\boxtimes (\sigma\widetilde{W}\otimes \widetilde{V_1}))\longrightarrow S((\sigma\widetilde{W}\otimes \widetilde{V_1})\boxtimes \widetilde{W^*})\\
    & \mathcal{C}(\mathbf{a'}): S((\sigma\widetilde{W}\otimes \widetilde{V_1})\boxtimes \widetilde{W^*})\longrightarrow S(\widetilde{V_2}).
\end{align*}
\end{lemma}
\begin{proof}
 Consider the following diagram
\begin{equation}
\begin{tikzcd}[sep=huge]
S(\widetilde{X}) \arrow[r,"\mathcal{C}(\mathbf{b})"] \arrow[dr, bend right ,"\mathcal{C}((\sigma f_2\circ \textnormal{id}\otimes f_1)\circ \mathbf{b})"'] & S(\widetilde{\sigma W_1}\otimes\widetilde{Y}\otimes \widetilde{W_2}) \arrow[r,"\mathbb{D}\Gamma_{F^{-1}}"] \arrow[d,"\mathcal{C}(\sigma f_1\otimes \textnormal{id} \otimes f_2)"] & S(\widetilde{Y}\otimes\widetilde{W_2}\otimes \widetilde{W_1}) \arrow[rd, bend left,"\mathcal{C}(\mathbf{b'}\circ (\textnormal{id}\otimes f_2\otimes f_1))"] \arrow[d,"\mathcal{C}(\textnormal{id}\otimes f_2\otimes f_1)"] \\
 & S(\widetilde{\sigma W_1'}\otimes\widetilde{Y}\otimes \widetilde{W_2'}) \arrow[r,"\mathbb{D}\Gamma_{F^{-1}}"]  & S(\widetilde{Y}\otimes\widetilde{W_2'}\otimes\widetilde{W_1'}) \arrow[r,swap,"\mathcal{C}(\mathbf{b'})"] & S(\widetilde{Y})\end{tikzcd}.
\end{equation}
The bent triangles on the left and right are clearly commutative by Corollary $2.7$. It suffices to prove that the rectangle in the middle is commutative. But this is a direct consequence of \cite[Lemma 6.1.13]{xiao2017cycles}.

Let $X=V_1,Y=\mathbf{1}$, $W_1=W_1'=W^*$,  $W_2=\sigma W\otimes V_1, W_2'=W\otimes V_2$. Write $\mathbf{a''}$ for the image of $\mathbf{a}$ under the canonical isomorphism $\textnormal{Hom}(\sigma W\otimes V_1\otimes W^*,V_2)\cong \textnormal{Hom}(\sigma W\otimes V_1,W\otimes V_2)$. Take $\mathbf{b}=\delta_{\sigma W}\otimes \textnormal{id}$, $f_1=\textnormal{id}$, and $f_2=\mathbf{a''}$. Then the second assertion follows from the above commutative diagram.
\end{proof}

\begin{lemma}
For any $\alpha\in \hg(\widetilde{V_1},\widetilde{V_2})$, the construction of $\sv$ is independent from the choice of
\begin{itemize}
    \item[(1)] projective $\Lambda$-modules $W\in \textnormal{Rep}_{\Lambda}(\hat{G}_{\Lambda})$,
    \item [(2)] $\mathbf{a} \in\hg (V_1,\sigma W^*\otimes V_2\otimes W )$, such that $\Theta_{W}(\mathbf{a})=\alpha$,
    \item [(3)] $(V_1\otimes V_2)\otimes W\boxtimes W^*$-acceptable integers $(m_1,n_1,m_2,n_2)$.
\end{itemize}
\end{lemma}
\begin{proof}
The proof is completely similar to that of \cite[Lemma 6.2.5]{xiao2017cycles}, and we briefly discuss it here.

We start by proving the independence of $(3)$. Choose another quadruple of $(V_1\otimes V_2)\otimes W\boxtimes W^*$-acceptable integers $(m_1',n_1',m_2',n_2')\geq (m_1,n_1,m_2,n_2)$. We have the following diagram of Hecke correspondences
$$
\begin{tikzcd}[sep=huge]
\sht^{\loc(m_1',n_1')}_{V_1} \arrow[r,leftarrow] \arrow[d,"\res^{m_1'n_1'}_{m_1,n_1}"] &  \sht^{\lambda,\loc(m_1',n_1')}_{V_1\mid V_2} \arrow[r] \arrow[d,"\res^{m_1'n_1'}_{m_1,n_1}"] & 
\sht^{\loc(m_2',n_2')}_{V_2} \arrow[d, "\res^{m_2'n_2'}_{m_2,n_2}"] \\
\sht^{\loc(m_1,n_1)}_{V_1} \arrow[r,leftarrow] & \sht^{\lambda,\loc(m_1,n_1)}_{V_1\mid V_2} \arrow[r] & \sht^{\loc(m_2,n_2)}_{V_2}.
\end{tikzcd}
$$
This is the upper face of diagram $(4.5)$. As we discussed in \S $3$, all the vertical maps are smooth, the two squares are commutative, and the left square is Cartesian. Then $\mathcal{C}^{\loc (m_1',n_1')}_{\sch}(\mathbf{a})$ equals the pullback of $\mathcal{C}^{\loc (m_1,n_1)}_{\sch}(\mathbf{a})$ along the vertical maps.

Next, we prove the independence of $(1)$ and $(2)$ simultaneously. Consider that $\hat{G}$ acts on the filtration of $\mathcal{O}_G$ by right regular representation. Then $\mathcal{O}_G$ is realized as an ind-object in $\textnormal{Rep}_{\Lambda}(\hat{G})$. Let $X\in \textnormal{Rep}_{\Lambda}(\hat{G})$ be a projective object and we denote by $\underline{X}$ the underlying $\Lambda$-module of $X$ equipped with the trivial $\hat{G}$-action. Consider the following $\hat{G}_{\Lambda}$-equivariant maps
$$
    \textnormal{a}_X:X\rightarrow \mathcal{O}_G\otimes \underline{X}, \textnormal{ }x\mapsto \textnormal{a}_X(x)(g):=gx,
$$
$$
    m_X:\underline{X}^*\otimes X\rightarrow\mathcal{O}_G,\textnormal{ }(x^*,x)\mapsto m_{X}(x^*,x)(g):=x^*(gx),
$$
where we identify $\mathcal{O}_G\otimes \underline{X}$ as the space of $\underline{X}$-valued functions on $\hat{G}$ in the definition of $\textnormal{a}_X$ and $m_X$. Taking $X=W$, we have the following $\hat{G}\times \hat{G}$-module maps
$$
\sigma W^*\otimes V_2\otimes  W\xrightarrow{\textnormal{a}_{\sigma W^*}}\underline{W}^*\otimes \sigma \mathcal{O}_G\otimes V_2\otimes W \xrightarrow{m_{W}} \sigma\mathcal{O}_G\otimes V_2\otimes \mathcal{O}_G.
$$
The map $\hat{G}\times \hat{G}\rightarrow \hat{G}\sigma,\textnormal{ }(g_1,g_2)\mapsto \sigma (g_1)^{-1}\sigma (g_2)\sigma$ induces a natural map $d_{\sigma}:\Lambda[\hat{G}\sigma]\rightarrow \sigma \mathcal{O}_G\otimes \mathcal{O}_G$ which intertwines the $\sigma$-twisted conjugation action on $\Lambda[\hat{G}\sigma]$ and the diagonal action of $\hat{G}$ on $\sigma \mathcal{O}_G\otimes \mathcal{O}_G$. For any $\alpha\in \hg (V_1,\mathcal{O}_G\otimes V_2)$, denote by $\alpha'$ the image of $\alpha$ under the following map
$$\hg (V_1,\mathcal{O}_G\otimes V_2)\xrightarrow{d_{\sigma}} \hg(V_1,\sigma\mathcal{O}_G\otimes V_2\otimes \mathcal{O}_G).
$$
Direct computation yields the followings
$$
(m_{W}\circ \textnormal{a}_{\sigma W^*})\circ \mathbf{a'}=d_{\sigma}(\alpha'):V_1\rightarrow \sigma \mathcal{O}_G\otimes V_2\otimes \mathcal{O}_G,$$
and 
$$
\textnormal{id}_{V_2}\otimes e_W=\textnormal{ev}_{(1,1)}\circ (m_{W}\circ a_{W^*}):V_2\otimes W\otimes W^*\rightarrow V_2,
$$
where $\textnormal{ev}_{(1,1)}$ denotes the evaluation at $(1,1)\in\hat{G}\times \hat{G}$. In Lemma $5.2$, let $W_1\otimes W_2:=W\otimes W^*$, $W_1'\otimes W_2':=\mathcal{O}_G\otimes \mathcal{O}_G$, $f_1\otimes f_2:=m_{W}\circ a_{W^*}$, $\mathbf{b}:=\mathbf{a'}$, and $\mathbf{b'}:=\textnormal{ev}_{(1,1)}$. Then we have 
\begin{align*}
    \mathcal{C}_{\sch}(\mathbf{a}) & =\mathcal{C}(\textnormal{id}_{V_2}\otimes e_W)\circ \mathbb{D}\Gamma ^{*}_{F^{-1}_{(V_2\otimes \sch)\boxtimes \sch^*}}\circ \mathcal{C}(\mathbf{a'})\\
    &=\mathcal{C}(\textnormal{ev}_{(1,1)})\circ  \mathbb{D}\Gamma ^{*}_{F^{-1}_{(V_2\otimes \mathcal{O}_G)\boxtimes \mathcal{O}_G}}\circ \mathcal{C}(d_{\sigma}(\alpha')).
\end{align*}
We see from the last equality in the above that $\mathcal{C}_{\sch}(\mathbf{a})$ depends only on $\alpha$ and the lemma is thus proved.
\end{proof}

We claim that our construction of $\sv$ is compatible with the composition of morphisms. More precisely, we have the following lemma.
\begin{lemma}
For any $V_1,V_2,V_3\in\textnormal{Rep}_{\Lambda}(\hat{G}_{\Lambda})$ which are projective $\Lambda$-modules and $S_1,S_2\in\textnormal{Rep}_{\Lambda}(\hat{G}_{\Lambda})$, we have the following commutative diagram
\begin{equation}
	\begin{tikzcd}[column sep=small]
		\textnormal{Hom}_{\textnormal{Coh}^{\hat{G}}(\hat{G}\sigma)}(\widetilde{V_1},\widetilde{V_2}) \otimes \textnormal{Hom}_{\textnormal{Coh}^{\hat{G}}(\hat{G}\sigma)}(\widetilde{V_2},\widetilde{V_3}) \arrow[r,"\phi"] & \textnormal{Hom}_{\textnormal{Coh}^{\hat{G}}(\hat{G}\sigma)}(\widetilde{V_1},\widetilde{V_3}) \\
		\hg(\sigma S_1\otimes V_1\otimes S_1^*,V_2)\otimes \hg(\sigma S_2\otimes V_2\otimes S_2^*,V_3) \arrow[u,""] \arrow[d,"\mathcal{C}_{S_1}\otimes \mathcal{C}_{S_2}"] \arrow[r,"\phi'"] & \hg(\sigma S_2\otimes \sigma S_1\otimes V_1\otimes S_1^*\otimes S_2^*,V_3) \arrow[u,""] \arrow[d,"\mathcal{C}_{S_1\otimes S_2}"] \\
		\cc_{\sht^{\loc}}(S(\widetilde{V_1}),S(\widetilde{V_2}))\otimes \cc_{\sht^{\loc}}(S(\widetilde{V_2}),S(\widetilde{V_3}))\arrow[r,"\phi''"] & \cc_{\sht^{\loc}}(S(\widetilde{V_1}),S(\widetilde{V_3})),
	\end{tikzcd}
\end{equation}

where
\begin{itemize}
\item the unlabelled vertical arrows are given by the Peter-Weyl theorem
    \item $\phi$ is the compositions of morphisms in $\textnormal{Coh}^{\hat{G}}(\hat{G}\sigma)$
    \item $\phi ''$ is the composition described in \S $3.2$
    \item 
$\phi'(\mathbf{a_1}\otimes \mathbf{a_2})$ is defined to be the homomorphism
$$
    \sigma S_2\otimes \sigma S_1\otimes V_1\otimes S_1^*\otimes S_2^* \xrightarrow{\textnormal{id}_{\sigma S_2}\otimes\mathbf{a_1}\otimes \textnormal{id}_{ S_2^*}} \sigma S_2\otimes V_2\otimes S_2^*\xrightarrow{\mathbf{a_2}} V_3.
$$
\end{itemize}
\end{lemma}
\begin{proof}
The lemma can be proved following the same idea in the proof of \cite[Lemma 6.2.7]{xiao2017cycles}.
\end{proof}
\subsubsection{}
We study the endomorphism ring of the unit object in $\textnormal{P}(\sht^{\loc}_{\bar{k}},\Lambda)$. This will be used to prove the "$S=T$" theorem for Shimura sets in \S$12.3$.

Let $\delta_{\mathbf{1}}$ denote the intersection cohomology sheaf $\textnormal{IC}_0$ on  $\sht_0^{\loc(m,n)}$. The group theoretic description of the moduli of restricted local Shtukas (cf. \cite[\S 5.3.2]{xiao2017cycles}) implies that $\sht_0^{\loc(m,n)}$ is perfectly smooth. Thus $\delta_{\mathbf{1}}$ may be realized as 
$$
\delta_{\mathbf{1}}^{m,n}:=\Lambda \langle (m-n)\dim G \rangle\in \textnormal{P}(\sht^{\loc(m,n)}_{0},\Lambda)
$$
for every $m\geq n$. Fix a square root $q^{1/2}$. We write $E$ for $\mathbb{Z}_\ell$
\begin{corollary}
\begin{itemize}
    \item[(1)] There is a natural isomorphism $$
    \cc_{\sht^{\loc}}(\delta_{\mathbf{1}},\delta_{\mathbf{1}})\simeq \mathcal{H}_{G,E}
    $$
    where $\mathcal{H}_{G,E}$ denotes the Hecke algebra $C_c^{\infty}(G(\mathcal{O})\backslash G(F)/G(\mathcal{O}),E)$.
\item[(2)] We denote the map 
$$
\mathcal{S}_{\mathcal{O}_{[\hat{G}\sigma/\hat{G}]},\mathcal{O}_{[\hat{G}\sigma/\hat{G}]}}:
\textnormal{End}_{\textnormal{Coh}^{\hat{G}}(\hat{G}\sigma)}(\mathcal{O}_{[\hat{G}\sigma/\hat{G}]})\rightarrow \cc_{\sht^{\loc}}(\delta_{\mathbf{1}},\delta_{\mathbf{1}})
$$
by $\mathcal{S}_{\mathcal{O}}$ for simplicity.
Under the isomorphism in $(1)$, the map $\mathcal{S}_{\mathcal{O}}\otimes\textnormal{id}_{E[q^{-1/2},q^{1/2}]}$
coincides with the classical Satake isomorphism.
\end{itemize}

\end{corollary}
\begin{proof}
Recall the definition of the Borel-Moore homology $\textnormal{H}^{\textnormal{BM}}_i(X)$ for a perfect pfp algebraic space which is defined over an algebraically closed field (cf. \cite[A.1.3]{xiao2017cycles}). Assume $X_1$ and $X_2$ to be perfectly smooth algebraic spaces of pure dimension. Let $X_1\leftarrow C\rightarrow X_2$ be a correspondence. Then
\begin{align}
     & \cc_C\big((X_1,E\langle d_1 \rangle),(X_2,E\langle d_2 \rangle)\big) \\ \notag
     = & \textnormal{Hom}_{D_b^c(C,E)}\big(E\langle d_1 \rangle, \omega_C\langle d_2-2\dim X_2\big >   \big) \\ \notag
    = & \textnormal{H}^{\textnormal{BM}}_{2\dim X_2+d_1-d_2}(C).\notag
\end{align}
Then if $2\dim C=2\dim X_2+d_1-d_2$, the cohomological correspondences from $(X_1,E\langle d_1 \rangle)$ to $(X_2,E\langle d_2 \rangle)$ can be identified as the set of irreducible components of $C$ of maximal dimension.

For a perfect pfp algebraic space $X$ of dimension $d$, define $I$ to be the set of top-dimensional irreducible components of $X$. Then $\textnormal{H}^{\textnormal{BM}}_d(I)$ is the free $E$-module generated by the $d$-dimensional irreducible components of $X$, and thus can be identified with the space $C(I,E)$ of $E$-valued functions on $I$. The map $f\mapsto \sum_{C_i\in I}f(C_i)[C_i]$ establishes a bijection 
\begin{equation}
C(I,E)=\textnormal{H}^{\textnormal{BM}}_d(X). 
\end{equation}

With the above preparations, we get an isomorphism
\begin{equation}
    \mathcal{H}_{G,E}\simeq \cc_{\sht^{\loc}}(\delta_{\mathbf{1}},\delta_{\mathbf{1}}),
\end{equation}
via a similar argument as for \cite[Proposition 5.4.4]{xiao2017cycles}, and we finish the proof of $(1)$.

To prove part $(2)$, we first note that the statement holds for $E=\ql$ by \cite[Theorem 6.0.1(2)]{xiao2017cycles}. We sketch the proof here. Let $\mu$ be a central minuscule dominant coweight, and $\nu$ be a dominant coweight such that $\sigma(\nu)=\nu$. Choose $(m_1,n_1,m_2,n_2)$ to be $(\nu+\mu,\nu)$-acceptable. Take $\mathbf{a}\in\hg(V_{\nu}\otimes V_{\mu}\otimes V_{\nu^*},V_{\mu})$ to be the map induced by the evaluation map $\mathbf{e}_{\nu}:V_{\nu}\otimes V_{\nu^*}\rightarrow \mathbf{1}$.  
Consider the following diagram
$$
\begin{tikzcd}
\textnormal{pt}\arrow[r,leftarrow] & \textnormal{Gr}_{\leq\nu^*}\arrow[r,"\Delta"] & \textnormal{Gr}_{\leq\nu^*}\times \textnormal{Gr}_{\leq\nu^*}\arrow[r,leftarrow,"\sigma\times \textnormal{id}"]& \textnormal{Gr}_{\mu^*}\times \textnormal{Gr}_{\leq\mu^*}\arrow[r,leftarrow,"\Delta"] & \textnormal{Gr}_{\leq\nu^*} \arrow[r] & \textnormal{pt}.
\end{tikzcd}
$$
Recall the cohomological correspondences $\delta_{\textnormal{Ic}_{\nu^*}}$ and $e_{\textnormal{Ic}_{\nu^*}}$ defined in \cite[\S A.2.3.4]{xiao2017cycles}. Then 
$$
\mathcal{C}^{\loc(m_1,n_1)}_{V_{\nu}}(\mathbf{a})=\delta_{\textnormal{IC}_{\nu^*}}\circ \Gamma_{\sigma\times \textnormal{id}}^*\circ e_{\textnormal{Ic}_{\nu^*}}\in \textnormal{H}^{\textnormal{BM}}_{0}(\textnormal{Gr}_{\nu^*}(k)),
$$
 and the cohomological correspondence $\mathcal{C}^{\loc(m_1,n_1)}_{V_{\nu}}(\mathbf{a})$ can be identified with the function $f$ on $\textnormal{Gr}_{\nu^*}(k)$ whose value at $x\in \textnormal{Gr}_{\nu^*}(k)$ is given by $\textnormal{tr}(\phi_x\mid \textnormal{Sat}(V_{\nu^*})_{\bar{x}})$. Then up to a choice of $q^{1/2}$, the map $S_{\mathcal{O},\mathcal{O}}\otimes_{\ql}\textnormal{id}_{\ql[q^{1/2},q^{-1/2}]}$ coincides with the classical Satake isomorphism.

Now we come back to the case $E=\mathbb{Z}_{\ell}$. Write $Q$ for $\ql[q^{1/2},q^{-1/2}]$. The above argument shows that
$$
\mathcal{S}_{\mathcal{O}}\otimes Q:\textnormal{End}_{\textnormal{Coh}^{\hat{G}}(\hat{G}\sigma)}(\mathcal{O}_{[\hat{G}\sigma/\hat{G}]})\otimes_{\mathbb{Z}_{\ell}}Q\rightarrow \cc_{\sht^{\loc}_{\bar{k}}}(\delta_{\mathbf{1}},\delta_{\mathbf{1}})\otimes_{\mathbb{Z}_{\ell}}Q
$$
coincide with the classical Satake isomorphism. Note that
$$
\textnormal{End}_{\textnormal{Coh}^{\hat{G}}(\hat{G}\sigma)}(\mathcal{O}_{[\hat{G}\sigma/\hat{G}]})\otimes_{\mathbb{Z}_{\ell}}Q\simeq \mathbb{Z}_{\ell}[\hat{G}]^{\hat{G}}\otimes_{\mathbb{Z}_{\ell}}Q,
$$
where $\hat{G}$ acts on $\hat{G}$ by the $\sigma$-twisted conjugation. Considering the Satake transfer of the image of $\mathbb{Z}_{\ell}$-basis of  $\mathbb{Z}_{\ell}[\hat{G}]^{(\hat{G})}$ in $\mathbb{Z}_{\ell}[\hat{G}]^{(\hat{G})}\otimes_{\mathbb{Z}_{\ell}}Q$, we conclude the proof of $(2)$.
\end{proof}

\section{Cohomological correspondences between Shimura varieties}
In this section, we adapt the machinery developed in previous sections and apply it to the study of the cohomological correspondences between different Hodge type Shimura varieties following the idea of \cite{xiao2017cycles}.

\subsection{Preliminaries}
Let $(G,X)$ be a Shimura datum. For the rest of this paper, $E$ will denote the reflex field of $(G,X)$ (cf.\cite{milne2005introduction}). Let $K\subset G(\mathbb{A}_f)$ be a (sufficiently small) open compact subgroup and denote by $\textnormal{Sh}_K(G, X)$ the corresponding Shimura variety defined over $E$. Fix a prime $p > 2$ such that $K_p$ is a
hyperspecial subgroup of $G(\mathbb{Q}_p)$. We write $\underline{G}$ for the reductive group which extends $G$ to $\mathbb{Z}_{(p)}$ and such that $\underline{G}(\mathbb{Z}_p)=K_p$. Choose $\nu$ to be a place of $E$ lying over $p$. We write $\mathcal{O}_{E,(\nu)}$ for the localization of $\mathcal{O}_E$ at $\nu$. Results of Kisin \cite{kisin2010integral} and Vasiu \cite{vasiu2007good} state that for any Hodge type Shimura datum $(G,X)$, there is a smooth integral canonical model $\mathcal{S}_K(G,X)$
of $\textnormal{Sh}_K(G, X)$, which is defined over $\mathcal{O}_{E,(\nu)}$. Let $k_{\nu}$
denote the residue field of $\mathcal{O}_{E,\nu}$ and fix an algebraic closure $\bar{k}_{\nu}$ of $k_{\nu}$. We denote by
$\textnormal{Sh}_{\mu,K} := (\mathcal{S}_K (G,X)\otimes k_{\nu})^{\textnormal{pf}}$
the perfection of the special fiber of $\mathcal{S}_K(G,X)$. The perfection of mod $p$ fibre of Shimura varieties and moduli of local Shtukas are related by a map $\loc_p:\textnormal{Sh}_{\mu,K}\rightarrow \sht_{\mu}^{\loc}$. The construction of $\loc_p$ is via a $\underline{G}$-torsor over the crystalline cite $(\mathcal{S}_{K,k_{\nu}}/\mathcal{O}_{E,\nu})_{\textnormal{CRIS}}$ and we refer to \cite[\S 7.2.1]{xiao2017cycles} for a detailed discussion. In the Siegel case, it may be understood as the perfection of the morphism sending an abelian variety to its underlying $p$-divisible group. We need the following result of Xiao-Zhu \cite[Proposition 7.2.4]{xiao2017cycles} for our proof of the main theorem.
\begin{proposition}
Let $(m,n)$ be a pair of $\mu$-large integers. The morphism 
$$
\loc_p(m,n):=\res_{m,n}\circ \loc_p:\textnormal{Sh}_{\mu}\rightarrow \sht^{\loc (m,n)}_{\mu}
$$
is perfectly smooth.
\end{proposition}
\subsection{\'{E}tale local systems}
Let $\ell\neq p$ be a prime number. Assume that $\rho:G\rightarrow GL_{\ql}(W)$ is a $\ql$-representation of $G$. If $K\subset G(\mathbb{A}_f)$ is sufficiently small, we associate an \'{e}tale local system $\mathcal{L}_{\ell,W}$ on $\textnormal{Sh}_{\mu,K}$ to $W$ following the idea of \cite[\S 4]{liu2017rigidity} and \cite[\S III.6]{milne1990canonical} as follows.

Write $K=K_{\ell}K^{\ell}$ with $K_{\ell}\subset G(\ql)$ and $K^{\ell}\subset G(\mathbb{A}_f^{\ell})$. The representation $\rho$ restricts to a representation
$$
\rho_{K_{\ell}}:K(\ql)\rightarrow G(\ql)\rightarrow GL(W_{\ql}).
$$
Note that $K(\ql)$ is compact, and there exists a lattice $\Lambda_{W,\ell}\subset W_{\ql}$ fixed by $K(\ql)$. Now we vary the levels at $\ell$. Define 
$$
K_{\ell}^{(n)}:=K_{\ell}\cap \rho_{K(\ql)}^{-1}(\{g\in GL(\Lambda_{W,\ell})\mid g\equiv 1 \mod \ell^{n}   \}).
$$
Then we get a system of open neighborhoods of $1\in G(\ql)$. For each $n$, the construction of $K_{\ell}^{(n)}$ gives rise to a representation
$$
\rho_{K_{\ell}}^n: K_{\ell}/K_{\ell}^{(n)}\rightarrow GL(\Lambda_{W,\ell}/\ell^n\Lambda_{W,\ell}).
$$
The natural projection map $\textnormal{Sh}_{\mu,K_{\ell}^{(n)}K^{\ell}}\rightarrow \textnormal{Sh}_{\mu,K_{\ell}K^{\ell}}$ is a finite \'{e}tale cover with the group of deck transformations being $K_{\ell}/K_{\ell}^{(n)}$. Then the trivial   \'{e}tale $\mathbb{Z}/\ell^n\mathbb{Z}$-local system $\textnormal{Sh}_{\mu,K_{\ell}^{(n)}K^{\ell}} \times \Lambda_{W,\ell}/\ell^n\Lambda_{W,\ell}$ on $\textnormal{Sh}_{\mu,K_{\ell}^{(n)}K^{\ell}}$ gives rise to the \'{e}tale $\mathbb{Z}/\ell^{n}\mathbb{Z}$-local system
$$
\mathcal{L}_{W,\ell,n}:=\textnormal{Sh}_{\mu,K_{\ell}^{(n)}K^{\ell}} \times^{K_{\ell}/K_{\ell}^{(n)}} \Lambda_{W,\ell}/\ell^n\Lambda_{W,\ell}.
$$
Let 
\begin{equation}
    \mathcal{L}_{W,\mathbb{Z}_{\ell}}:=\varprojlim_{n}\mathcal{L}_{W,\ell,n}.
\end{equation}
This is an \'{e}tale $\mathbb{Z}_{\ell}$-local system on $\textnormal{Sh}_{\mu,K}$. It can be checked that $\mathcal{L}_{W,\ql}:=\mathcal{L}_{W,\mathbb{Z}_{\ell}}\otimes \mathbb{Q}$ is an  \'{e}tale $\ql$-local system on $\textnormal{Sh}_{\mu,K}$ which is independent of the choice of $\Lambda_{\ell}$.

\subsection{Main theorem}
Let $(G_1, X_1)$ and $(G_2, X_2)$ be two Hodge type Shimura data (cf. \cite{milne2005introduction}) equipped with an isomorphism $\theta:G_{1,\mathbb{A}_f}\simeq G_{2,\mathbb{A}_f}$. Let
$\{\mu_i\}$ denote the conjugacy class of Hodge cocharacters determined by $X_i$ and consider them as dominant characters of $\hat{T}$. In particular, $\mu_1$ and $\mu_2$ are both minuscule. Then \cite[Corollary 2.1.5]{xiao2017cycles} implies that there is a canonical inner twist $\Psi_{\mathbb{R}}:G_1\rightarrow G_2$ over $\mathbb{C}$. Recall notations in \S 1.3. We define $\mu_{i,\textnormal{ad}}$ to be the composition of $\mu_i$ with the quotient $G\rightarrow G_{\textnormal{ad}}$ and consider it as a character of $\hat{T}_{\textnormal{sc}}$. We assume that 
$$
\mu_{1,\textnormal{ad}}\mid _{Z(\hat{G}_{\textnormal{sc}}^{\Gamma_{\mathbb{Q}}})}=\mu_{2,\textnormal{ad}}\mid _{Z(\hat{G}_{\textnormal{sc}}^{\Gamma_{\mathbb{Q}}})}.
$$
It follows from \cite[Corollary 2.1.6]{xiao2017cycles} that $\Psi_{\mathbb{R}}$ comes from a unique global inner twist $\Psi:G_{1\bar{\mathbb{Q}}}\rightarrow G_{2\bar{\mathbb{Q}}}$ such that $\Psi=\textnormal{Int}(h)\circ \theta$, for some $\theta:G_{1,\mathbb{A}_f}\simeq G_{2,\mathbb{A}_f}$ and $h\in G_{2,\textnormal{ad}}(\bar{\mathbb{A}}_f)$.

We assume that $K_{i}\subset G(\mathbb{A}_f)$ to be sufficiently small such that $\theta K_1=K_2$.  Choose a prime $p$ such that $K_{1,p}$ (and therefore $K_{2,p}$) is hyperspecial. Let $\underline{G_i}$ be the integral model of $G_{i,\mathbb{Q}_p}$
over $\mathbb{Z}_p$ determined by $K_{i,p}$. Then  $\underline{G_1}\simeq \underline{G_2}$, and we can thus identify
their Langlands dual groups $(\hat{G},\hat{B},\hat{T})$. Choose an isomorphism $\iota:\mathbb{C}\simeq \bar{\mathbb{Q}}_p$. Let $\nu\mid p$ be a place of the compositum of reflex fields of $(G_i,X_i)$ determined
by our choice of isomorphism $\iota$. We write $\textnormal{Sh}_{\mu_i}$ for the mod $p$ fibre of the canonical integral model of $\textnormal{Sh}_{K_i}(G_i,X_i)$ base change to $k_{\nu}$. We make the following assumption
\begin{equation}
    \mu_{1}\mid _{Z(\hat{G}_{}^{\Gamma_{\mathbb{Q}_p}})}= \mu_{2}\mid _{Z(\hat{G}_{}^{\Gamma_{\mathbb{Q}_p}})}.
\end{equation}
The assumption guarantees the existence of the ind-scheme $\textnormal{Sh}_{\mu_1\mid\mu_2}$ which fits into the following commutative diagram
\begin{equation}
    \begin{tikzcd}[sep=huge]
\textnormal{Sh}_{\mu_1,K_1} \arrow[r,leftarrow, "\overleftarrow{h}_{\mu_1}"] \arrow[d,"\loc_p"] & \textnormal{Sh}_{\mu_1\mid\mu_2} \arrow[r,rightarrow,"\overrightarrow{h}_{\mu_2}"] \arrow[d] & \textnormal{Sh}_{\mu_2,K_2} \arrow[d,"\loc_p"] \\
\sht^{\loc}_{\mu_1}  \arrow[r,leftarrow, "\overleftarrow{h}_{\mu_1}^{\loc}"] & \sht^{\loc}_{\mu_1\mid\mu_2} \arrow[r,rightarrow, "\overrightarrow{h}_{\mu_2}^{\loc}"] & \sht^{\loc}_{\mu_2}
\end{tikzcd},
\end{equation}
making both squares to be Cartesian. 
\begin{remark}
In the case that $(G_1,X_1)=(G_2,X_2)$, $\textnormal{Sh}_{\mu_1\mid\mu_2}$ is the perfection of the mod p fibre of
a natural integral model of some Hecke correspondence. If $(G_1,X_1)\neq (G_2,X_2)$, then $\textnormal{Sh}_{\mu_1\mid\mu_2}$ can be regarded as “exotic Hecke correspondences” between mod p fibres of different Shimura varieties. We refer to \cite[\S7.3.3,\S 7.3.4]{xiao2017cycles} for a detailed discussion.
\end{remark}

Let $(G_i,X_i)$ $i=1,2,3$ be three Hodge type Shimura data, together with the isomorphisms $\theta_{i,j}:G_{i,\mathbb{A}_f}\simeq G_{j,\mathbb{A}_f}$ satisfying the natural cocycle condition. Choose a common level $K$ using the isomorphism $\theta_{i,j}$. Let $p$ be an unramified prime,
such that the assumption $(12.2)$ holds for each pair of $((G_i,X_i),(G_j,X_j))$. Choose a half Tate twist $\ql(1/2)$.

Let $V_i:=V_{\mu_i}$ be the highest weight representation of $\hat{G}_{\ql}$ of highest weight $\mu_i$. Write $\widetilde{V_i}\in \textnormal{Coh}^{\hat{G}_{\ql}}(\hat{G}_{\ql}\sigma)$ for the vector bundle associated to $V_i$ analogous to \S $5.1.3$. Recall from \S $6.1$ that, to each representation $W$ of $G_{\ql}$, we can attach the \'etale local system $\mathcal{L}_{W,\ql}$ on $\textnormal{Sh}_{\mu_i}$. Let $d_i=\langle  2\rho,\mu_i \rangle=\dim \textnormal{Sh}_K(G_i,X_i)$. Denote the global section of the structure sheaf on the quotient stack $[\hat{G}\sigma/\hat{G}]$ by $\mathcal{J}$, and the prime-to-$p$ Hecke algebra by $\mathcal{H}^p$.
\begin{theorem}
There exists a map 
\begin{equation}
  \textnormal{Spc}:\textnormal{Hom}_{\textnormal{Coh}^{\hat{G}_{\ql}}(\hat{G}_{\ql}\sigma)}(\widetilde{V_1},\widetilde{V_2})\rightarrow  \textnormal{Hom}_{\mathcal{H}^p\otimes \mathcal{J}}(\textnormal{H}_c^{*}(\textnormal{Sh}_{\mu_1},\mathcal{L}_{W,\ql}\langle d_1 \rangle), \textnormal{H}_c^{*}(\textnormal{Sh}_{\mu_2},\mathcal{L}_{W,\ql}\langle d_2 \rangle),
\end{equation}
which is compatible with compositions on the source and target. 
\end{theorem}
\begin{proof}
Choose a lattice $\Lambda_i\in\textnormal{Rep}_{\mathbb{Z}_{\ell}}(\hat{G}_{\mathbb{Z}_{\ell}})$ in $V_i$. We denote by $\widetilde{\Lambda_i}\in\textnormal{Coh}^{\hat{G}_{\mathbb{Z}_{\ell}}}(\hat{G}_{\mathbb{Z}_{\ell}}\sigma)$ the coherent sheaf which corresponds to $\Lambda_i$ as in \S $5.1$. Then
\begin{align}
    \textnormal{Hom}_{\textnormal{Coh}^{\hat{G}_{\ql}}(\hat{G}_{\ql}\sigma)}(\widetilde{V_1},\widetilde{V_2}) & \simeq \textnormal{Hom}_{\hat{G}_{\ql}}(V_1,V_2\otimes \ql[\hat{G}])\\ \notag
    &\simeq \textnormal{Hom}_{\hat{G}_{\ql}} (\Lambda_1\otimes_{\mathbb{Z}_{\ell}}\ql,(\Lambda_2\otimes_{\mathbb{Z}_{\ell}} \mathbb{Z}_{\ell}[\hat{G}])\otimes_{\mathbb{Z}_{\ell}}\ql)\\ \notag
    & \simeq \textnormal{Hom}_{\hat{G}_{\mathbb{Z}_{\ell}}}(\Lambda_1,\Lambda_2\otimes_{\mathbb{Z}_{\ell}}\mathbb{Z}_{\ell}[\hat{G}])\otimes_{\mathbb{Z}_{\ell}}\ql\\ \notag
    &  \simeq \textnormal{Hom}_{\textnormal{Coh}^{\hat{G}_{\mathbb{Z}_{\ell}}}(\hat{G}_{\mathbb{Z}_{\ell}}\sigma)}(\widetilde{\Lambda_1},\widetilde{\Lambda_2})\otimes_{\mathbb{Z}_{\ell}}\ql.
\end{align}
By Theorem $5.1$, we get a map
\begin{equation}
\mathcal{S}_{\Lambda_1,\Lambda_2}:\textnormal{Hom}_{\textnormal{Coh}^{\hat{G}_{\mathbb{Z}_{\ell}}}(\hat{G}_{\mathbb{Z}_{\ell}}\sigma)}(\widetilde{\Lambda_1},\widetilde{\Lambda_2})\rightarrow \cc_{\sht^{\loc}}( S(\widetilde{\Lambda_1}), S(\widetilde{\Lambda_2})). 
\end{equation}
Combining $(6.5)$ with $(6.6)$, we get the following map
\begin{equation}
    \textnormal{Hom}_{\textnormal{Coh}^{\hat{G}_{\ql}}(\hat{G}_{\ql}\sigma)}(\widetilde{V_1},\widetilde{V_2})\rightarrow \cc_{\sht^{\loc}}( S(\widetilde{\Lambda_1}), S(\widetilde{\Lambda_2}))\otimes_{\mathbb{Z}_{\ell}}\ql.
\end{equation}
Choose a dominant coweight $\nu$ and a quadruple $(m_1, n_1, m_2, n_2)$ that is $(\mu_1 + \nu, \nu)$-
acceptable and $(\mu_2 + \nu, \nu)$-acceptable. We have the following diagram 
\begin{equation}
    \begin{tikzcd}[sep=huge]
\textnormal{Sh}_{\mu_1} \arrow[r,leftarrow, "\overleftarrow{h}_{\mu_1}"] \arrow[d,"\loc_p"] & \textnormal{Sh}_{\mu_1\mid\mu_2}^{\nu}\arrow[d,"\loc_p^{\nu}"] \arrow[r,rightarrow,"\overrightarrow{h}_{\mu_2}"] \arrow[d] & \textnormal{Sh}_{\mu_2} \arrow[d,"\loc_p"] \\
\sht^{\loc}_{\mu_1}  \arrow[r,leftarrow, "\overleftarrow{h}_{\mu_1}^{\loc}"] \arrow[d,"\res_{m_1,n_1}"] & \sht^{\nu,\loc}_{\mu_1\mid\mu_2} \arrow[r,rightarrow, "\overrightarrow{h}_{\mu_2}^{\loc}"] \arrow[d,"\res^{\nu}_{m_1,n_1}"] & \sht^{\loc}_{\mu_2} \arrow[d,"\res_{m_2,n_2}"]\\
\sht^{\loc(m_1,n_1)}_{\mu_1} \arrow[r,leftarrow, "\overleftarrow{h}_{\mu_1}^{\loc (m_1,n_1)}"] & \sht^{\nu,\loc(m_1,n_1)}_{\mu_1\mid\mu_2} \arrow[r,rightarrow, "\overrightarrow{h}_{\mu_2}^{\loc(m_2,n_2)}"] & \sht^{\loc (m_2,n_2)}_{\mu_2}
\end{tikzcd},
\end{equation}
where
\begin{itemize}
    \item all squares are commutative (discussions on diagram $(4.5)$ and diagram $(6.3)$,
    \item except for the square at the down right corner, and the other three squares are Cartesian (discussions on diagram $(6.3)$ and diagram $(4.5)$, 
    \item the morphism $\overleftarrow{h}_{\mu_1}$ is perfectly proper (\cite[Lemma 5.2.12]{xiao2017cycles}),
    \item the morphisms $\loc_p(m_i,n_i)$ are perfectly smooth (Proposition $6.1$).
\end{itemize}
Then the morphism $\loc_p^{\nu}(m_1,n_1):=\res_{m_1,n_1}^{\nu}\circ \loc_p^{\nu}$ is also perfectly proper. Thus we can pullback the cohomological correspondences (cf. \cite[A.2.11)]{xiao2017cycles}) on the right hand side of $(6.6)$ along $\loc_p^{\nu}(m_1,n_1)$ to obtain a map
$$
\loc_p^{\nu}(m_1,n_1)^{\star}:\cc_{{\sht}^{\loc}}(S(\widetilde{\Lambda_1}),S(\widetilde{\Lambda_2}))\rightarrow\cc_{\textnormal{Sh}_{\mu\mid\mu}^{\nu}}(\loc_p(m_1,n_1)^{\star}S(\widetilde{\Lambda_1}),   \loc_p(m_2,n_2)^{\star}(S(\widetilde{\Lambda_2})).
$$
Note that $\mu_i$ are minuscule, then the $\star$-pullback of $S(\widetilde{\Lambda_i})$ along $\loc_p(m_i,n_i)$ equals $\mathbb{Z}_{\ell}\langle d_i  \rangle$. Next, we construct a natural map 
\begin{equation}
    \mathfrak{C}_{W}:\cc_{\textnormal{Sh}_{\mu_1\mid\mu_2}^{\nu}}\big ((\textnormal{Sh}_{\mu_1},\mathbb{Z}_{\ell}\langle d_1  \rangle),(\textnormal{Sh}_{\mu_2},\mathbb{Z}_{\ell}\langle d_2  \rangle)\big )\rightarrow \cc_{{\textnormal{Sh}_{\mu_1\mid\mu_2}^{\nu}}}\big ((\textnormal{Sh}_{\mu_1},\mathcal{L}_{W,\mathbb{Z}_{\ell}}\langle d_1 \rangle),(\textnormal{Sh}_{\mu_2},\mathcal{L}_{W,\mathbb{Z}_{\ell}}\langle d_2 \rangle)\big).
\end{equation}
For each $n\in\mathbb{Z}^+$, we note that there exists an ind-scheme $\textnormal{Sh}_{\mu_1\mid\mu_2}^{(n)}$ which fits into the following commutative diagram such that both squares are Cartesian 
$$
\begin{tikzcd}[sep=huge]
\textnormal{Sh}_{\mu_1,K_{\ell}^{(n)}K^{\ell}} \arrow[r,leftarrow, "\overleftarrow{h}_{\mu_1}^{(n)}"] \arrow[d,"p_1^n"] & \textnormal{Sh}_{\mu_1\mid\mu_2}^{\nu,(n)} \arrow[r,rightarrow,"\overrightarrow{h}_{\mu_2}^{(n)}"] \arrow[d,"p^n"] & \textnormal{Sh}_{\mu_2,K_{\ell}^{(n)}K^{\ell}} \arrow[d,"p_2^n"] \\
\textnormal{Sh}_{\mu_1}  \arrow[r,leftarrow, "\overleftarrow{h}_{\mu_1}"] & \textnormal{Sh}_{\mu_1\mid\mu_2}^{\nu} \arrow[r,rightarrow, "\overrightarrow{h}_{\mu_2}"] & \textnormal{Sh}_{\mu_2}.
\end{tikzcd}
$$
Here the three vertical maps are the natural quotients by the finite group $K_{\ell}/K_{\ell}^{n}$ and are thus \'{e}tale.

Let $(f_n)_n:(\overleftarrow{h}_{\mu_1})^*(\mathbb{Z}/\ell^n\mathbb{Z}\langle d_1 \rangle)_n\rightarrow (\overrightarrow{h}_{\mu_2})^!(\mathbb{Z}/\ell^n\mathbb{Z}\langle d_2 \rangle)_n$ be a cohomological correspondence in $\cc_{\textnormal{Sh}_{\mu_1\mid\mu_2}^{\nu}}\big ((\textnormal{Sh}_{\mu_1},\mathbb{Z}_{\ell}\langle d_1  \rangle),(\textnormal{Sh}_{\mu_2},\mathbb{Z}_{\ell}\langle d_2  \rangle)\big )$. For each $n\in\mathbb{Z}^{+}$, the shifted pullback (cf. \cite[A.2.12]{xiao2017cycles}) of $f_n$ gives rise to a cohomological correspondence 
$$
\tilde{f}_n:(\overleftarrow{h}_{\mu_1}^{(n)})^*(\mathbb{Z}/\ell^n\mathbb{Z}\langle d_1 \rangle)\rightarrow (\overrightarrow{h}_{\mu_2}^{(n)})^!(\mathbb{Z}/\ell^n\mathbb{Z}\langle d_2 \rangle)
$$
in $ \cc_{\textnormal{Sh}_{\mu_1\mid\mu_2}^{\nu,(n)}}\big ((\textnormal{Sh}_{\mu_1,K_{\ell}^{(n)}K^{\ell}},\mathbb{Z}/{\ell^n}\mathbb{Z}\langle d_1  \rangle),(\textnormal{Sh}_{\mu_2,K_{\ell}^{(n)}K^{\ell}},\mathbb{Z}/{\ell^n}\mathbb{Z}\langle d_2  \rangle)\big )$. For any representation $W$ of $G_{\ql}$, recall the $\mathbb{Z}/\ell^n\mathbb{Z}$ module $\Lambda_{W,\ell}/\ell^n\Lambda_{W,\ell}$ constructed in \S $6.2$. The cohomological correspondence $\tilde{f}_n$ gives rise to a cohomological correspondence 
$$
\tilde{g}_n\in \cc_{\textnormal{Sh}_{\mu_1\mid \mu_2}^{\nu,(n)}}(\textnormal{Sh}_{\mu_1,K_{\ell}^{(n)}K^{\ell}}\times \Lambda_{W,\ell}/\ell^n\Lambda_{W,\ell}\langle d_1\rangle ,\textnormal{Sh}_{\mu_2,K_{\ell}^{(n)}K^{\ell}}\times \Lambda_{W,\ell}/\ell^n\Lambda_{W,\ell}\langle d_2 \rangle).
$$
In addition, the cohomological correspondence $\tilde{f}_n$ is $K_\ell/K_\ell^{(n)}$-equivariant. Then it follows that the cohomological correspondence $\tilde{g}_n$ is also $K_\ell/K_\ell^{(n)}$-equivariant and descends to a cohomological correspondence
$$
g_n\in \cc_{\textnormal{Sh}_{\mu_1\mid\mu_2}^{\nu}}((\textnormal{Sh}_{\mu_1},\mathcal{L}_{W,\ell,n}\langle d_1  \rangle),(\textnormal{Sh}_{\mu_2},\mathcal{L}_{W,\ell,n}\langle d_2  \rangle)).
$$

Defining $\mathfrak{C}_W((f_n)_n):=(g_n)_n$ completes the construction of $\mathfrak{C}_W$.

Compose the maps we previously construct,
\begin{align}
  &  \cc_{\sht^{\nu,\loc(m_1,n_1)}_{\mu_1\mid\mu_2}}\Big ((\sht^{\loc(m_1,n_1)}_{\mu_1},S(\widetilde{\Lambda_1})^{\loc (m_1,n_1)}),(\sht^{\loc(m_2,n_2)}_{\mu_1},S(\widetilde{\Lambda_2})^{\loc (m_2,n_2)})\Big )\\ \notag
  \xrightarrow{\loc_p^{\nu}(m_1,n_1)^{\star}} & \cc_{\textnormal{Sh}_{\mu_1\mid\mu_2}^{\nu}}((\textnormal{Sh}_{\mu_1},\mathbb{Z}_{\ell}\langle d_1  \rangle),(\textnormal{Sh}_{\mu_2},\mathbb{Z}_{\ell}\langle d_2  \rangle))\\ \notag
  \xrightarrow{\mathfrak{C}_W} & \cc_{{\textnormal{Sh}_{\mu_1\mid\mu_2}^{\nu}}}((\textnormal{Sh}_{\mu_1},\mathcal{L}_{W,\mathbb{Z}_{\ell}}\langle d_1 \rangle),(\textnormal{Sh}_{\mu_2},\mathcal{L}_{W,\mathbb{Z}_{\ell}}\langle d_2 \rangle))\\ \notag
  \xrightarrow{\textnormal{H}_c^*} & \textnormal{Hom}_{\mathcal{H}^p}(\textnormal{H}_c^{*}(\textnormal{Sh}_{\mu_1},\mathcal{L}_{W,\mathbb{Z}_{\ell}}\langle d_1 \rangle),\textnormal{H}_c^{*}(\textnormal{Sh}_{\mu_2},\mathcal{L}_{W,\mathbb{Z}_{\ell}}\langle d_2 \rangle)).
\end{align} 
We justify that the composition of maps in $(6.10)$ factors through $\cc_{\sht^{\loc}}(S(\widetilde{V_1}),S(\widetilde{V_2}))$. Note that the proof of Lemma $5.3.(3)$ and the definition of $\loc_p^{\nu}(m_1,n_1)^{\star}$ imply that for a quadruple $(m_1',n_1',m_2',n_2')$ of $(\mu_1+\nu,\nu)$-acceptable and $(\mu_2+\nu,\nu)$-acceptable integers, the functor $\loc_p^{\nu}(m_1,n_1)^{\star}$ commutes with the connecting morphism in $(4.12)$ (with $\mu_1$,$\mu_2$,$\lambda$ fixed). Let $\nu\leq\nu'$ and $(m_1',n_1',m_2',n_2')$ be a quadruple of non-negative integers satisfying appropriate acceptance conditions. The proper smooth base change shows that $\loc_p^{\nu}(m_1',n_1')^{\star}$ commutes with enlarging $\nu$ to $\nu'$. In addition, the proper smooth base change together with the construction of $\mathfrak{C}_W$ show that the following diagram commutes:
$$
\begin{tikzcd}[row sep=huge]
\cc_{\textnormal{Sh}_{\mu_1\mid\mu_2}^{\nu}}((\textnormal{Sh}_{\mu_1},\mathbb{Z}_{\ell}\langle d_1  \rangle),(\textnormal{Sh}_{\mu_2},\mathbb{Z}_{\ell}\langle d_2  \rangle)) \arrow[r,"\mathfrak{C}_W"] \arrow[d,"i_*"] & \cc_{{\textnormal{Sh}_{\mu_1\mid\mu_2}^{\nu}}}((\textnormal{Sh}_{\mu_1},\mathcal{L}_{W,\mathbb{Z}_{\ell}}\langle d_1 \rangle),(\textnormal{Sh}_{\mu_2},\mathcal{L}_{W,\mathbb{Z}_{\ell}}\langle d_2 \rangle)) \arrow[d,"i_*"] \\
\cc_{\textnormal{Sh}_{\mu_1\mid\mu_2}^{\nu'}}((\textnormal{Sh}_{\mu_1},\mathbb{Z}_{\ell}\langle d_1  \rangle),(\textnormal{Sh}_{\mu_2},\mathbb{Z}_{\ell}\langle d_2  \rangle)) \arrow[r,"\mathfrak{C}_W"] & \cc_{{\textnormal{Sh}_{\mu_1\mid\mu_2}^{\nu'}}}((\textnormal{Sh}_{\mu_1},\mathcal{L}_{W,\mathbb{Z}_{\ell}}\langle d_1 \rangle),(\textnormal{Sh}_{\mu_2},\mathcal{L}_{W,\mathbb{Z}_{\ell}}\langle d_2 \rangle)).
\end{tikzcd}
$$
Thus the map $\mathfrak{C}_W$ is compatible with the enlargement of $\nu$. Finally, by \cite[Lemma A.2.8]{xiao2017cycles}, the composition of maps $\textnormal{H}_c^{*}\circ \mathfrak{C}_W$ commutes with enlarging $\nu$ to $\nu'$. We complete the proof of the statement at the beginning of this paragraph.

Composing $(6.7)$ with $(6.11)$, we get a canonical map
\begin{equation}
    \textnormal{Hom}_{\textnormal{Coh}^{\hat{G}_{\ql}}(\hat{G}_{\ql}\sigma)}(\widetilde{V_1},\widetilde{V_2})\rightarrow \textnormal{Hom}_{\mathcal{H}^p}(\textnormal{H}_c^{*}(\textnormal{Sh}_{\mu_1},\mathcal{L}_{W,\ql}\langle d_1 \rangle),\textnormal{H}_c^{*}(\textnormal{Sh}_{\mu_2},\mathcal{L}_{W,\ql}\langle d_2 \rangle)).
\end{equation}
The fact that $(6.9)$ is compatible with the compositions of the source and target can be proved in an analogous way as \cite[Lemma 7.3.12]{xiao2017cycles}, and we omit the details. Then the action of $\mathcal{J}$ naturally translates to the right hand side of $(6.9)$ and upgrades it to our desired map
$$
\textnormal{Spc}:\textnormal{Hom}_{\textnormal{Coh}^{\hat{G}}(\hat{G}\sigma)}(\widetilde{V_1},\widetilde{V_2})\rightarrow  \textnormal{Hom}_{\mathcal{H}^p\otimes \mathcal{J}}(\textnormal{H}_c^{*}(\textnormal{Sh}_{\mu_1},\mathcal{L}_{W,\ql}\langle d_1 \rangle), \textnormal{H}_c^{*}(\textnormal{Sh}_{\mu_2},\mathcal{L}_{W,\ql}\langle d_2 \rangle)).
$$

\end{proof}
As discussed in \textit{loc.cit}, the action of $\mathcal{J}$ on $\textnormal{H}_c^{*}(\textnormal{Sh}_{\mu_i},\mathcal{L}_{W,\ql}\langle d_i \rangle)$ is expected to coincide with the usual Hecke algebra action, which may be understood as the Shimura variety analogue of V. Lafforgue's “$S=T$” theorem (cf. \cite{lafforgue2018chtoucas}). We prove this in the case of Shimura sets.
\begin{proposition}
Let $\textnormal{Sh}_K(G,X)$ be a zero-dimensional Shimura variety. Then the action of $\mathcal{J}$ on $\textnormal{H}_c^{*}(\textnormal{Sh}_{\mu_i},\mathcal{L}_{W,\ql}\langle d_i \rangle)$ is given by the classical Satake isomorphism.
\end{proposition}
\begin{proof}
Let $f\in\mathcal{J}$. Since the Shimura variety we consider is zero-dimensional, it follows from \cite[A.2.3(5)]{xiao2017cycles} that the cohomological correspondence $\loc_p^{\star}(\mathcal{S}_{\mathcal{O}}(f))$ can be identified with a $\mathbb{Z}_{\ell}$-valued function on $\textnormal{Sh}_{\mu\mid\mu}$. By our construction of the map $\textnormal{Spc}$, this function is given by the pullback of a function $f'$ on $\sht^{\loc}_{\mu\mid\mu}=G(\mathbb{Z}_p)\backslash G(\mathbb{Q}_p)/G(\mathbb{Z}_p)$. Corollary $5.5(2)$ thus implies that the function $f'$ is exactly the function $\mathcal{S}_{\mathcal{O}}(f)\in H_{G,]\mathbb{Z}_\ell[p^{-1/2},p^{1/2}]}$ which is the image of $f$ under the classical Satake isomorphism.

For any $n\in \mathbb{Z}^+$, take $W=\mathbb{Z}_{\ell}^n$. Recall our construction of $\mathfrak{C}_W$, the cohomological correspondence $\tilde{f}_n$ is given by a finite direct sum of the function $\loc_p^{\star}(\mathcal{S}_{\mathcal{O}}(f))$ since the Shimura variety we consider is a set of discrete points. Then the action of $\textnormal{Spc}(f)$ on $\textnormal{H}_c^*(\textnormal{Sh}_{\mu_i},\mathcal{L}_{W,\ql}\langle d_i \rangle)$ is given by the classical Satake isomorphism. For general $W$, we take resolutions of it as in $(6.10)$, and the statement follows from the case $W=\mathbb{Z}_{\ell}^n$.
\end{proof}
\subsection{Non-vanishing of the geometric Jacquet-Langlands transfer}
In Theorem $6.2.2$, we constructed the geometric Jacquet-Langlands transfer 
$$
\textnormal{Spc}:\textnormal{Hom}_{\textnormal{Coh}^{\hat{G}_{\ql}}(\hat{G}_{\ql}\sigma)}(\widetilde{V_1},\widetilde{V_2})\rightarrow  \textnormal{Hom}_{\mathcal{H}^p\otimes \mathcal{J}}(\textnormal{H}_c^{*}(\textnormal{Sh}_{\mu_1},\mathcal{L}_{W,\ql}\langle d_1 \rangle), \textnormal{H}_c^{*}(\textnormal{Sh}_{\mu_2},\mathcal{L}_{W,\ql}\langle d_2 \rangle).
$$

It is natural to ask when this transfer map is nonzero. We discuss this issue in this section. The idea essentially follows from the discussion in \cite[\S 7.4]{xiao2017cycles}, and we briefly sketch it here.

Assume that $\textnormal{Sh}_{\mu_1,K_1}(G_1,X_1)$ is a zero dimensional Shimura variety. The Jacquet-Langlands transfer map induces the following map 
$$
    \textnormal{JL}_{1,2}(\mathbf{a}):\textnormal{H}_c^0(\textnormal{Sh}_{\mu_1},\mathcal{L}_{W,\mathbb{Q}_\ell})\rightarrow \textnormal{H}_c^*(\textnormal{Sh}_{\mu_2},\mathcal{L}_{W,\mathbb{Q}_\ell}\langle d_2\rangle),
$$
for $\mathbf{a}\in\textnormal{Hom}_{\textnormal{Coh}^{\hat{G}_{\ql}}(\hat{G}_{\ql}\sigma_p)}(\widetilde{V_1},\widetilde{V_2})$. Let $\mathbf{a}'\in \textnormal{Hom}_{\textnormal{Coh}^{\hat{G}_{\ql}}(\hat{G}_{\ql}\sigma_p)}(\widetilde{V_2},\widetilde{V_1})$be the morphism such that the induced map 
$$
    \textnormal{JL}_{2,1}(\mathbf{a}'):\textnormal{H}_c^*(\textnormal{Sh}_{\mu_2},\mathcal{L}_{W,\mathbb{Q}_\ell}\langle d_2\rangle)\rightarrow \textnormal{H}_c^0(\textnormal{Sh}_{\mu_1},\mathcal{L}_{W,\mathbb{Q}_\ell})
$$
is dual to $\textnormal{JL}_{1,2}(\mathbf{a})$ when viewing it as a cohomological correspondence (cf. \cite[\S A.2.18]{xiao2017cycles}).
 
The composition map $\textnormal{JL}_{2,1}(\mathbf{a}')\circ \textnormal{JL}_{1,2}(\mathbf{a})$ gives rise to an endomorphism of $\widetilde{V_{\mu_1}}\in \textnormal{Coh}^{\hat{G}_{\ql}}(\hat{G}_{\ql}\sigma_p)$.
 By \cite[Theorem 1.4.1]{xiao2017cycles}, the hom spaces $\textnormal{Hom}_{\textnormal{Coh}^{\hat{G}_{\ql}}(\hat{G}_{\ql}\sigma)}(\widetilde{V_1},\widetilde{V_2})$ and $\textnormal{Hom}_{\textnormal{Coh}^{\hat{G}_{\ql}}(\hat{G}_{\ql}\sigma)}(\widetilde{V_2},\widetilde{V_1})$ are both finite projective $\mathcal{J}$-modules. Thus it makes sense to consider the determinant of the pairing
 \begin{equation}
   \textnormal{Hom}_{\textnormal{Coh}^{\hat{G}_{\ql}}(\hat{G}_{\ql}\sigma_p)}(\widetilde{V_2},\widetilde{V_1}) \otimes \textnormal{Hom}_{\textnormal{Coh}^{\hat{G}_{\ql}}(\hat{G}_{\ql}\sigma_p)}(\widetilde{V_1},\widetilde{V_2})\rightarrow \mathcal{J}.
 \end{equation}
 In particular, this determinant can be regarded as a regular function on the stack $[\hat{G}_{\ql}\sigma_p/\hat{G}_{\mathbb{Q}_\ell}]$; for a detailed discussion on the pairing $(6.12)$, see \cite{xiao2018vector}.

By Theorem $6.1.2$ in \textit{loc.cit}, we conclude the following result:
\begin{theorem}
Let $\pi_f$ be an irreducible $\mathcal{H}_K$-module, and let 
$$
\textnormal{H}_c^0(\textnormal{Sh}_{\mu_1},\mathcal{L}_{W,\mathbb{Q}_\ell})[\pi_f]:=\textnormal{Hom}_{\mathcal{H}_K}(\pi_f,\textnormal{H}_c^0(\textnormal{Sh}_{\mu_1},\mathcal{L}_{W,\mathbb{Q}_\ell}))\otimes \pi_f
$$
denote the $\pi_f$-isotypical component.
Then, the map 
$$
\textnormal{JL}_{1,2}(\mathbf{a}):\textnormal{H}_c^0(\textnormal{Sh}_{\mu_1},\mathcal{L}_{W,\mathbb{Q}_\ell})\rightarrow \textnormal{H}_c^d(\textnormal{Sh}_{\mu_2},\mathcal{L}_{W,\mathbb{Q}_\ell})
$$
restricted to $\textnormal{H}_c^0(\textnormal{Sh}_{\mu_1},\mathcal{L}_{W,\mathbb{Q}_\ell})[\pi_f]$ is injective if the Satake parameters of $\pi_f$ is general with respect to $V_{\mu_2}$ in the sense of \cite{xiao2017cycles}.
\end{theorem}

\providecommand{\bysame}{\leavevmode\hbox to3em{\hrulefill}\thinspace}
\providecommand{\MR}{\relax\ifhmode\unskip\space\fi MR }
\providecommand{\MRhref}[2]{%
  \href{http://www.ams.org/mathscinet-getitem?mr=#1}{#2}
}
\providecommand{\href}[2]{#2}

\end{document}